\documentclass[10pt]{amsart}
\usepackage{stmaryrd}
\usepackage{tikz}
\usetikzlibrary{arrows,positioning,shapes,matrix,backgrounds}
\usepackage{float}
\usepackage{amssymb}
\usepackage{amsthm}
\usepackage{enumerate}
\usepackage{url}
\usepackage{amsmath}
\usepackage{graphicx}
\usepackage{textcomp}
\usepackage{marvosym}
\usepackage{bbm}
\usepackage{fancybox}
\usepackage{fancyheadings}
\usepackage[section]{placeins}
\usepackage{setspace}
\newcommand{\dsp}{\omega\times 2^\omega}

\newcommand{\gA}{\mathfrak{A}}
\newcommand{\ind}{\mathrm{Ind}}

\newcommand{\gf}{\mathfrak{f}}

\newcommand{\gC}{\mathfrak{C}}
\newcommand{\gB}{\mathfrak{B}}

\newcommand{\g}[1]{\mathfrak{#1}}

\newcommand{\zfc}{\mathrm{ZFC}}
\newcommand{\meag}{\mbox{\tt meagre}}
\newcommand{\pathh}{\mbox{\tt path}}
\newcommand{\nulli}{\mbox{\tt null}}

\newcommand{\func}[2]{{^{#1}}#2}

\newcommand{\scr}{\mathbb}

\newcommand{\bld}{\textbf}

\newcommand{\cT}{{\mathcal T}}

\newcommand{\cD}{{\mathcal D}}
\newcommand{\cE}{{\mathcal E}}
\newcommand{\cA}{{\mathcal A}}
\newcommand{\cI}{{\mathcal I}}

\newcommand{\cV}{{\mathcal V}}
\newcommand{\cP}{{\mathcal P}}
\newcommand{\cM}{{\mathcal M}}

\newcommand{\cC}{{\mathcal C}}

\newcommand{\cU}{{\mathcal U}}
\newcommand{\bc}{\mathrm{BC}}

\newcommand{\cB}{{\mathcal B}}

\newcommand{\Ra}{\rangle}
\newcommand{\force}{\Vdash}
\newcommand{\La}{\langle}

\newcommand{\eb}[1]{\emph{\bld{#1}}}

\newcommand{\fr}[1]{\mathrm{Fr} \, {#1}}
\newcommand{\sfr}[1]{\mathrm{Fr}^*  {#1}}

\newtheorem{thm}{Theorem}[section]

\newtheorem{df}[thm]{Definition}

\newtheorem{ineq}{Inequality}

\newtheorem{exmm}[thm]{Example}

\newtheorem{fact}[thm]{Fact}

\newtheorem{rem}[thm]{Remark}
\newtheorem{lem}[thm]{Lemma}
\newtheorem{prop}[thm]{Proposition}
\numberwithin{equation}{section}

\long\def\symbolfootnote[#1]#2{\begingroup%
\def\thefootnote{\fnsymbol{footnote}}\footnote[#1]{#2}\endgroup} 
\def\@fnsymbol#1{\ifcase#1\hbox{}\or *\or \dagger\or \ddagger\or \mathchar "278\or \mathchar "27B\or \|\or **\or \dagger\dagger \or \Ankh \else\@ctrerr\fi\relax} 
\tikzset{
  every node/.style={scale=1}
}
\tikzset{
    >=latex,
    punkt/.style={
           rectangle,
           rounded corners,
           draw=black, very thick,
           text width=10em,
           minimum height=2em,
           text centered},
    pil/.style={
           ->,
           ultra thick,
           shorten <=10pt,
           shorten >=10pt,}
}
\tikzset{
  every node/.style={scale=1}
}

\usepackage{t1enc}
\usepackage[latin2]{inputenc}
\usepackage{amsmath}
\usepackage{amsthm}
\usepackage{amssymb}
\usepackage[T1]{fontenc}
\usepackage{enumerate}
\usepackage{verbatim}
\usepackage{amscd}
\usepackage{color}

\usepackage[colorlinks]{hyperref}
\usepackage[displaymath, tightpage]{preview}
\usepackage[colorinlistoftodos, textwidth=4cm, shadow]{todonotes}

\begin{document}

\title{On submeasures on Boolean algebras}
\thanks{Date: 31/12/2012}
\thanks{The results presented here appear in the author's doctoral thesis \cite{phd1}.}
\author[Omar Selim]{Omar Selim}
\email{oselim.mth@gmail.com}
\begin{abstract} 
We present a collection of observations and results concerning submeasures on Boolean algebras. They are all motivated by Maharam's problem and Talagrand's construction that solved it.
\end{abstract}
\maketitle
\tableofcontents
\section{Introduction}\label{introduction}
\noindent
Given a Boolean algebra $\gA$, a function $\mu:\gA\rightarrow \scr{R}$ is called a \emph{submeasure} if the following conditions are satisfied:
\begin{itemize}
\item $\mu(0) = 0$;
\item $(\forall a,b\in \gA)(a\leq b \rightarrow \mu(a)\leq \mu(b));$
\item $(\forall a,b\in \gA)(\mu(a\cup b)\leq \mu(a)+\mu(b))$.
\end{itemize}
A submeasure is \emph{additive} if, given disjoint $a,b\in \gA$, we have $\mu(a\cup b) = \mu(a)+\mu(b)$. Additive submeasures are called \emph{measures}. Two submeasures $\mu$ and $\lambda$ are \emph{equivalant} if, for any sequence $(a_n)_n$ from $\gA$, we have 
$$
\lim_n\mu(a_n) = 0 \leftrightarrow \lim_n \lambda(a_n) = 0.
$$
A submeasure $\mu$ on $\gA$ is called \emph{exhaustive} if, given a pairwise disjoint sequence $(a_n)_n$ from $\gA$, we have $\lim_n\mu(a_n) = 0$. Maharam's problem, also known as the control measure problem, asks if every exhaustive submeasure is equivalent to a measure. This problem was first asked in \cite{maharam47} (but not in this form) and in \cite{tal08}, building on \cite{roberts91} and \cite{farah04}, M. Talagrand constructs an exhaustive submeasure on the clopen (closed and open) sets of the Cantor space that is not equivalent to a measure. Maharam's problem and the numerous forms in which it appears are, by now, well documented. A detailed treatment of this topic is given in \cite[Chapter 39]{fremv3}, and a very accessible survey, which also discusses a related problem of von Neumann, is given in \cite{balcar06}.\\\\
A frustrating aspect of Maharam's problem is the complexity of its solution. It seems, at least in the literature, that very little progress has been made (or perhaps, attempted) in trying to analyse Talagrand's construction. The only discussion (other than \cite{tal08}) on Talagrand's solution that we are aware of is \cite{frem08}. Consequently it is still not clear how much more insight we have into Maharam's problem now that it has been settled, and in particular why it was so difficult. What follows is a result of trying to understand Maharam's problem, its solution, and also some of the open questions that concern this solution.\\\\
This article may be summarised as follows. In Section \ref{oca} we discuss how one can generalise the main result of \cite{dow00} to show that, under Todorcevic's Open Colouring Axiom (OCA), the Boolean algebra $\cP(\omega)/\mathrm{Fin}$ does not contain a Maharam algebra as a subalgebra. This was motivated by the problem of whether or not Talagrand's construction of a non-measurable Maharam algebra contains a regular subalgebra isomorphic to the random algebra (see \cite[Question 12]{farah07}, \cite[Problem 3A]{frem08}, \cite[Question 3]{velickovic09}).\\\\
In Section \ref{forcingsubs} we consider the collection of submeasures $\mu:\gA\rightarrow \scr{Q}$, where $\gA$ is some finite subalgebra of the collection of clopen sets of the Cantor space, partially ordered by reverse inclusion. We show that any generic filter for this forcing gives rise to a submeasure that is not equivalent to a measure, but is exhaustive with respect to the antichains that appear in the ground model. This line of investigation is motivated by the fact that Maharam's problem is equivalent to a $\Pi^1_2$-statement and is thus absolute for forcing extensions (see \cite{balcar06}). In particular if one could force the existence of an exhaustive submeasure that is not equivalent to a measure, then the existence of such a submeasure follows from $\zfc$.\\\\
In Section \ref{forcingideal} we investigate some basic properties of the forcing associated to Talagrand's construction. We show that in any such forcing extension, the collection of random reals will have $\nu$-measure $0$, where $\nu$ is Talagrand's submeasure. We also give a proof that, in any such extension, the collection of ground model reals will be Lebesgue null and meagre.\\\\
In Section \ref{calcs} we consider the first pathological submeasure $\psi$ constructed by Talagrand in \cite{tal08}. We give explicit values for the entire space and for relative atoms. The motivation here is that the values of the Lebesgue measure on the Cantor space are indeed easily calculable, and it would be very helpful if the same could be said for Talagrand's submeasure.\\\\
In Section \ref{measures} we find a linear association between the space of all functionals on the clopen sets of the Cantor space and the collection of all signed measures on this space. We give an explicit construction of this map. This result may be interesting in itself, but perhaps it may also shed some light on the theory of submeasures by allowing us to consider the classical theory of measures.
\section{Preliminaries}\label{prels}
We gather here the background material that we shall need in what follows. We have tried to make our notation and terminology as standard as possible. Unless otherwise stated our set theory follows \cite{kunen80} and in particular `$p\leq q$' is taken to mean that `$p$ is stronger than $q$'. Everything concerning Boolean algebras follows \cite{hbba1}. By $\scr{A}$ we denote the countable atomless Boolean algebra. If $\gA$ is a Boolean algebra and $A\subseteq \gA$, then we let $\La A \Ra$ be the smallest subalgebra of $\gA$ containing $A$. If $a\in \gA$ then by $\gA_a$ we denote the Boolean algebra $\{b\in \gA:b\leq a\}$ (with unit $a$). We let $\scr{N} = \{1,2,...\}$ and $\omega = \{0,1,2,...\}$. If $n\in \scr{N}$ then we let $[n] = \{1,2,...,n\}$. Given sets $(X_i)_{i\in \scr{N}}$, $I\subseteq \scr{N}$ and $s\in \prod_{i\in I} X_i$, we let 
$$
[s] = \{f\in \prod_{i\in \scr{N}}X_i: (\forall i\in I)(f(i) = s(i))\}.
$$
Given a topological space $K$, we let $\mathrm{Clopen}(K)$ and $\mathrm{Borel}(K)$ be the collections of clopen sets and Borel sets of $K$, respectively.\\\\
Unless otherwise stated, everything concerning submeasures may be found in \cite{fremv3}. A \emph{functional} on a Boolean algebra $\gA$ is any real-valued function defined on $\gA$ that assumes the value $0$ at $0_\gA$. Given a submeasure $\mu$ on a Boolean algebra $\gA$, we say that $\mu$ is \emph{strictly positive} if $(\forall a\in \gA)(\mu(a) = 0 \rightarrow a = 0)$. The submeasure $\mu$ is called \emph{diffuse} if, for every $\epsilon > 0 $ we can find a finite partition $A$ of $\gA$, such that $\max\{\mu(a):a\in A\}<\epsilon$. The submeasure $\mu$ is called \emph{uniformly exhaustive} if, for every $\epsilon >0$, we can find an $N\in\scr{N}$ such that for any antichain $a_1,...,a_N$ from $\gA$ we have $\min_i \mu(a_i) < \epsilon$. The submeasure $\mu$ is called \emph{pathological} if there does not exist a non-zero measure $\lambda$ on $\gA$ such that $\lambda\leq \mu$, where by $\lambda\leq \mu$ we mean $(\forall a\in \gA)(\lambda(a)\leq \mu(a))$. The well-known Kalton-Roberts theorem reads as follows.
\begin{thm}\emph{(\cite{kalton83})} A submeasure is uniformly exhaustive if and only if it is equivalent to a measure.
\end{thm}
If $\gA$ is $\sigma$-complete then $\mu$ is called \emph{continuous} if for each sequence $a_1\geq a_2\geq \cdots$ from $\gA$ such that $\prod_i a_i = 0$ we have $\lim_i \mu(a_i) = 0$. It follows that if $\mu$ is continuous and $(a_i)_{i}$ is a sequence from $\gA$ such that 
$$
a:= \limsup_i a_i = \liminf_i a_i
$$ 
then $\lim_i\mu(a_i) = \mu(a)$. An atomless $\sigma$-complete Boolean algebra that carries a strictly positive continuous submeasure is called a \emph{Maharam algebra}. Every Maharam algebra satisfies the countable chain condition (ccc) and no such algebra can add a Cohen real (see \cite[Theorem 5.9]{balcar06}). If $\gA$ is $\sigma$-complete then $\mu$ is called \emph{$\sigma$-additive} if for every antichain $(a_i)_i$ from $\gA$ we have 
$$
\mu(\sum_i a_i) = \sum_i \mu(a_i).
$$
An atomless $\sigma$-complete Boolean algebra that carries a strictly positive $\sigma$-additive measure is called a \emph{measure algebra}. The \emph{random algebra} is the unique measure algebra that is $\sigma$-generated by a countable set. The \emph{Cohen algebra} is the unique atomless $\sigma$-complete Boolean algebra with a countable dense subset.\\\\
Since $\sigma$-additivity implies continuity, every measure algebra is a Maharam algebra. The different formulations of Maharam's problem that we shall need are as follows.
\begin{fact}\emph{(\cite[\S 393]{fremv3})} The following statements are equivalant.
\begin{itemize}
\item Every Maharam algebra is a measure algebra.
\item Every exhaustive submeasure is equivalent to a measure.
\item Every exhaustive submeasure is uniformly exhaustive.
\item Every exhaustive submeasure is not pathological.
\end{itemize}
\end{fact}
Finally, although our notation follows quite closely to that of \cite{tal08}, for completeness we present Talagrand's example of an exhaustive pathological submeasure that is not uniformly exhaustive. For the remainder of this section, everything is taken from \cite{tal08}. Let \index{$\cT$}
$$
\cT = \prod_{n\in\scr{N}}[2^n].
$$
We also fix\index{$\scr{T}$}
$$
\scr{T} = \mathrm{Clopen}(\cT).
$$
For each $n\in\scr{N}$, let $\cA_n = \{[f\restriction [n]]:f\in \cT\}$ and $\cB_n$\index{$\cB_n$} be the subalgebra of $\scr{T}$ generated by $\cA_n$\index{$\cA_n$}. Members of $\cB_n$ will be finite unions of sets of the form $[s]$, for $s\in \prod_{k\in[n]}[2^k]$.\\\\
Let
$$
\cM = \scr{T} \times [\scr{N}]^{<\omega}\times \scr{R}_{\geq 0}.
$$
For finite $X\subseteq \cM$, where $X = \{(X_1,I_1,w_1),...,(X_n,I_n,w_n)\}$,
let
$$
w(\emptyset) = 0, \ \ \ w(X) = \sum_{i=1}^n w_i, \ \ \ \bigcup X = \bigcup_{i=1}^n X_i.
$$
The value $w(X)$ is called the \bld{weight}\index{weight of a subset of $\cD$} of $X$.\\\\
We have the following general construction.
\begin{df} \label{generalsub}If $Y\subseteq \cM$ and is such that there exists a finite $Y'\subseteq Y$ such that $\cT = \bigcup Y'$ then $Y$ defines a submeasure $\phi_Y$\index{$\phi_Y$} given by
$$
\phi_Y(B) = \inf\{w(Y'):Y'\subseteq Y \wedge \mbox{$Y'$ is finite} \wedge B\subseteq \bigcup Y'\}.
$$
\end{df}
For $k\in\scr{N}$ and $\tau\in [2^n]$ let \index{$S_{n,\tau}$}
$$
S_{n,\tau} = \{f\in \cT:f(n)\neq \tau\}.
$$
For $k\in\scr{N}$ let\index{$\eta(k)$|}\index{$\alpha(k)$}
$$
\eta(k) = 2^{2k+10}2^{(k+5)^4}(2^3+2^{k+5}2^{(k+4)^4}), \ \alpha(k) = (k+5)^{-3}
$$ 
and set
$$
\cD_k = \{(X,I,w)\in \cM: |I|\in [\eta(k)]\wedge w = 2^{-k}\left(\frac{\eta(k)}{|I|}\right)^{\alpha(k)} \wedge (\exists \tau\in \prod_{n\in I}[2^n])(X = \bigcap_{n\in I} S_{n,\tau(n)})\}.
$$
Let $\cD = \bigcup_{k\in\scr{N}} \cD_k$\index{$\psi$} and 
$$
\psi = \phi_\cD.
$$
An important property of $\psi$ is the following.
\begin{prop}\label{onpsi} Any non-trivial submeasure $\mu$ such that $\mu\leq \psi$ must be pathological and cannot be uniformly exhaustive.
\end{prop}
Thus it is enough to now construct a non-trivial exhaustive submeasure that lies below $\psi$.
\begin{df}\label{274} Let $\mu:\scr{T}\rightarrow \scr{R}$ be a submeasure
and let $m,n\in\scr{N}$.
\begin{enumerate}[\hspace{0.0cm} $\bullet$]
\item For each $s\in \prod_{i\in [m]}[2^i]$ we define the
map 
$$
\pi_{[s]}: \cT \rightarrow [s]
$$
by 
$$
(\pi_{[s]}(x))(i) = \left\{\begin{array}{cl} s(i) ,\quad&\mbox{if $i\in [m]$;}\\
x(i),\quad&\mbox{otherwise.}\end{array}\right.
$$
\item For $m<n$ we say a set $X\subseteq \cT$ is \emph{\textbf{$(m,n,\mu)$-thin}}\index{thin@$(m,n)$-thin} if and only if
$$
(\forall A\in\cA_m)(\exists B\in\cB_n)(B\subseteq A \wedge B\cap
X=\emptyset \wedge \mu(\pi_A^{-1}[B])\geq 1).
$$
For $I\subseteq \scr{N}$, we say that $X$ is \textbf{\emph{$(I,\mu)$-thin}}\index{thin@$(I,\mu)$-thin} if it is $(m,n,\mu)$-thin for each $m,n\in I$ with $m<n$.
\end{enumerate}
\end{df}
The rest of the construction proceeds by a downward induction. For $p\in\scr{N}$ let $\cE_{p,p} = \cC_{p,p} = \cD$ and $\psi_{p,p} = \phi_{\cC_{p,p}}$. Now for $k<p$, given $\cE_{k+1,p}$, $\cC_{k+1,p}$ and $\psi_{k+1,p}$, we let 
$$
\cE_{k,p} = \{(X,I,w)\in\cM:\mbox{$X$ is $(I,\psi_{k+1,p})$-thin, $|I|\in [\eta(k)]$ and $w=2^{-k}\left(\frac{\eta(k)}{|I|}\right)^{\alpha(k)}$}\},
$$
$\cC_{k,p} = \cC_{k+1,p}\cup \cE_{k,p}$ and $\psi_{k,p} = \phi_{\cC_{k,p}}$.\\\\
Next let $\cU$ be a non-principal ultrafilter on $\scr{N}$. For each $k\in \scr{N}$ let $\cE_{k}$ and $\cC_{k}$\index{$\cC_k$}\index{$\cE_k$} be subsets of $\cM$ defined by
$$
x\in \cE_{k} \leftrightarrow \{p:x\in \cC_{k,p}\}\in \cU,
$$
and $\cC_{k} = \cD\cup \bigcup_{l\geq k} \cE_{l}$.\\\\
Finally, let $\nu_k = \phi_{\cC_k}$. It is clear from Definition \ref{generalsub} that we have 
$$
\nu_1\leq \nu_2\leq \nu_3 \cdots \leq \psi.
$$
Now the submeasure $\nu_1$, which we shall denote by $\nu$\index{$\nu$} from here on, is the desired counter example to Maharam's problem. The fact that $\nu$ is non-trivial and exhaustive requires two separate arguments. Exhaustivity follows by showing that for each $k$ and antichain $(a_n)_{n}$ from $\scr{T}$ we have 
$$
\limsup_n \nu_k(a_n)\leq 2^{-k}.
$$
This last property is known as \emph{$2^{-k}$-exhaustivity}.
\section{OCA and Maharam algebras} \label{oca}
The main result of \cite{dow00} is the following.
\begin{thm}\label{maindowhart} Assuming OCA, the random algebra is not a subalgebra of $\cP(\omega)/\mathrm{Fin}$.
\end{thm}
In this section we show how Theorem \ref{maindowhart} remains true if we replace the random algebra by any Maharam algebra:
\begin{thm}\label{myoca} Assuming OCA, the Boolean algebra $\cP(\omega)/\mathrm{Fin}$ does not contain a Maharam algebra as a subalgebra.
\end{thm}
\begin{df}\label{cb} Let $\gB$ be a $\sigma$-complete Boolean algebra. Say that $\scr{C}(\gB)$\index{$\scr{C}(\gB)$} holds if and only if there exists an embedding $\scr{F}:\mathrm{Clopen}(\omega\times {{^\omega}2}) \rightarrow \gB$\index{$F$} such that
\begin{equation*}
(\forall n\in\omega)(\forall f_1,f_2,... \in {{^\omega}2})(\exists N_1, N_2,...\in\omega)(\sum_{i\in\omega} \scr{F}(\{n\}\times [f_i\restriction N_i])< \scr{F}(\{n\}\times 2^\omega)).
\end{equation*}
\end{df}
Here we work in the product space $\dsp$ corresponding to the discrete topology on $\omega$. The passage from Theorem \ref{maindowhart} to Theorem \ref{myoca} relies on the observation that the proof of Theorem \ref{maindowhart} from \cite{dow00} still works if one replaces the random algebra by any $\sigma$-complete Boolean algebra $\gB$ such that $\scr{C}(\gB)$ holds. Since the arguments needed under these more general conditions are (verbatim) the same as the original ones, we refer the reader to \cite{dow00}.\\\\
What is left to do here then is to show that $\scr{C}(\gB)$ holds for any Maharam algebra $\gB$, and this is a straightforward consequence of continuity and the following fact.
\begin{fact}\label{diffusefact}\eb{(\cite[392Xg]{fremv3})} If $\gB$ is an atomless $\sigma$-complete Boolean algebra then any strictly positive continuous submeasure on $\gB$ will be diffuse.
\end{fact}
\begin{prop} $\scr{C}(\gB)$ holds for every (atomless) Maharam algebra $\gB$.
\end{prop}
\begin{proof} Let $\gB$ be a Maharam algebra carrying a strictly positive continuous submeasure $\mu$. Let $(a_i)_{i\in\omega}$ be a partition of $\gB$. Since $\mu$ will be diffuse, for each $i,j\in\omega$ we can find a partition $A_i^j$ of $a_i$ into finitely many pieces each with $\mu$-measure not greater than $\frac{1}{j+1}$. For each $i$, let $\gB_i$ be the (countable atomless) subalgebra of $\gB_{a_i}$ generated by $\bigcup_{j} A_{i}^j$ and let $f_i:\scr{A}\rightarrow \gB_i$ be any isomorphism. Now let 
$$
\scr{F}(\bigcup_{n\in \omega}\{n\}\times B_n) = \sum_{n\in \omega}f_n(B_n).
$$
Let $f\in \func{\omega}{2}$ and $m\in \omega$. For each $\epsilon>0$ there exists a finite partition $a_1,a_2,...,a_n$ of  $\mathrm{Clopen}(\{m\}\times 2^\omega)$ such that for each $i$, $\mu(\scr{F}(a_i))\leq \epsilon$. But for $k$ large enough there will be an $i$ such that $\{m\}\times [f\restriction k]\subseteq a_i$, and so $\mu(\scr{F}(\{m\}\times [f\restriction k]))\rightarrow 0$ as $k\rightarrow \infty$. Thus given $f_0,f_1,...\in \func{\omega}{2}$ for each $i$ we can choose $N_i$ such that $\mu(\scr{F}(\{m\}\times [f\restriction N_i]))<2^{-i-2}\mu(\scr{F}(\{m\}\times 2^\omega))$. Then by $\sigma$-subadditivity of $\mu$ we get 
$$
\mu(\sum_{i\in\omega} \scr{F}(\{n\}\times [f_i\restriction N_i]))\leq \sum_{i\in\omega}\mu( \scr{F}(\{n\}\times [f_i\restriction N_i]))< \frac{1}{2}\mu(\scr{F}(\{n\}\times [\emptyset]),
$$ 
and since $\mu$ is strictly positive we are done.\end{proof}

\section{Forcing submeasures}\label{forcingsubs}
Let $\scr{P}$ be the collection of all normalised submeasures $\mu:\gA\rightarrow [0,1]\cap\scr{Q}$ where $\gA$ is a finite subalgebra of $\scr{A}$. Order $\scr{P}$ by reverse inclusion: $\mu\preceq \lambda$ if and only if $\lambda\subseteq \mu$.\\\\
In this section we prove the following.
\begin{thm}\label{1611111} Let $M$ be a countable transitive model of of ZFC. If $G\subseteq \scr{P}\in M$ is $\scr{P}$-generic over $M$ then $\lambda:=\bigcup G$ is a normalised submeasure on $\scr{A}$ that is not uniformly exhaustive and is such that for any antichain $(a_i)_{i\in\scr{N}}\in M$ we have $\lim_i\lambda(a_i)= 0$.
\end{thm}
In fact $\scr{P}$ is a well known forcing notion.
\begin{lem}\label{itiscohen} The separative quotient of $\scr{P}$ is countably infinite and atomless and therefore its Boolean completion is the Cohen algebra.
\end{lem}
\begin{proof} Let $\scr{P}'$ be the separative quotient of $\scr{P}$. Since the submeasures in $\scr{P}$ only take rational values and we have assumed $\gB$ is countable the partial order $\scr{P}$ is also countable and so $\scr{P}'$ is at most countable. Given $a\in \gB^+\setminus \{{1}\}$ and $q\in \scr{Q}\cap [0,1]$ we can always find a submeasure $\lambda_q\in \scr{P}$ such that $\lambda_q(a) = q$.  The $\lambda_q$ correspond to countably many distinct equivalence classes of $\scr{P}'$, so $\scr{P}'$ is infinite. Now suppose that $\lambda\in \scr{P}$ and let $a$ be an atom of $\mathrm{dom}(\lambda)$ such that $\lambda(a)>0$.
Let $c\in\scr{A}^+$ be such that $c<a$ and let $\gA$ be the subalgebra generated by $\mathrm{dom}(\lambda)\cup \{c\}$. Let $\lambda_1$ be the submeasure on $\gA$ defined by 
\begin{equation}\label{extendbyone}
\lambda_1(b) = \min\{\lambda(d):d\in \gA\wedge b\subseteq d\}.
\end{equation}
Then $\lambda_1\preceq \lambda$ and $\lambda_1(c) = \lambda(a)$. Now given $b\in \gA$ we can find $b'\in \mathrm{dom}(\lambda)$ and $b''\in \{c,a\setminus c,0\}$ such that $b = b'\sqcup b''$, so we can take $\lambda_2$ to be the submeasure on $\gA$ defined by
$$
 \lambda_2(b) = \left\{\begin{array}{cl} \lambda(b'\cup a),\quad&\mbox{if $b'' = a\setminus c$;}\\
\lambda(b'),\quad&\mbox{otherwise.}\end{array}\right.
$$
Then $\lambda_2\preceq \lambda$ and $\lambda_2(c) = 0 \neq \lambda_1(c)$. Thus $\lambda_1$ and $\lambda_2$ correspond to two different members of $\scr{P}'$ and so one of them must define a different equivalence class to $\lambda$, and we are done.
\end{proof}
Thus Theorem \ref{1611111}, together with Lemma \ref{itiscohen}, is saying that in any forcing extension adding a Cohen real there exists a submeasure, constructed from $\scr{P}$, that is not uniformly exhaustive but is exhaustive with respect to the antichains from the ground model. Of course in any forcing extension due to $\scr{P}$ new antichains might be added that have not been accounted for. It is not clear to us how to proceed. A natural direction to pursue would be to attempt a forcing iteration of some sort. However it seems that the claims gathered here do not give us enough information about the added submeasure to do so.\\\\
Let us prove Theorem \ref{1611111}. The following is implicit in \cite[Theorem 2]{herer75}, and actually we have already used two instances of it in Lemma \ref{itiscohen}.
\begin{lem}\label{1611112} Let $\gC$ be an atomless Boolean algebra carrying a normalised submeasure $\mu$ and let $a_0,a_1,...,a_n\in \gC^+$ be a finite partition of $\gC$. Let $\varphi_0,...,\varphi_n$ be normalised submeasures on $\gC_{a_0},...,\gC_{a_n}$, respectively. Then $\gC$ carries a normalised submeasure $\varphi$ such that for each $a\in \bigcup_{i}\gC_{a_i} \cup \La a_0,...,a_n\Ra$, we have
\begin{equation}\label{fromchris}
\varphi(a) =  \left\{\begin{array}{cl} \mu(a_i)\varphi_i(a),\quad&\mbox{if $i\in n+1$ and $a\in\gC_{a_i}$;}\\
\mu(a),\quad&\mbox{otherwise.}\end{array}\right.
\end{equation}
Moreover, if the $\mu$ and the $\varphi_i$ take only rational values then so does $\varphi$.
\end{lem}
\begin{proof} Let $\gA = \La a_0,a_1,...,a_n\Ra$. Define the function $f:\bigcup_{i = 0}^n \gC_{a_i} \cup \gA\rightarrow \scr{R}$ by 
$$
f(a) =  \left\{\begin{array}{cl} \mu(a_i)\varphi_i(a),\quad&\mbox{if $i\in n+1$ and $a\in\gC_{a_i}$;}\\
\mu(a),\quad&\mbox{otherwise.}\end{array}\right.
$$
Now define $\varphi:\gC\rightarrow \scr{R}$ by
\begin{equation}\label{inftomin}
\varphi(a) = \inf\{\sum_{c\in A}f(c):A\in [\bigcup_{i = 0}^n \gC_{a_i}\cup \gA]^{<\omega} \wedge c\leq  \sum A\}.
\end{equation}
It is straightforward to check that $\varphi$ is a submeasure.\\\\
Finally let us observe that the value in (\ref{inftomin}) can always be achieved by a cover from $[\bigcup_{i = 0}^n \gC_{a_i}\cup \gA]^{<\omega}$, that is to say, in (\ref{inftomin}) we are actually taking a minimum rather than an infimum. In particular the submeasure $\varphi$ will be rational valued if $\mu$ and the $\varphi_i$ are. To obtain the minimum, note that if $a\in \gC$ then $a$ will be a disjoint union of sets $c_i\in \gC_{a_i}$ and so if $A\in [\bigcup_{i = 0}^n \gC_{a_i}\cup \gA]^{<\omega}$ is a cover of $a$ then we can find another cover $A'\in [\bigcup_{i = 0}^n \gC_{a_i}\cup \gA]^{<\omega}$ such that the members of $A'$ are pairwise disjoint, $(\forall i)(A'\cap \gC_{a_i} \in \{\emptyset, c_i\})$ and 
$$
\sum_{c\in A'}f(c)\leq \sum_{c\in A}f(c).
$$ 
Since there are only finitely many such $A'$ (given $a$) we are done.
\end{proof}
If $\gC$ is a Boolean algebra carrying a normalised submeasure $\mu:\gC\rightarrow [0,1]$ and $n\in \scr{N}$ then we will say that $\mu$ is \emph{$n$-pathological} if we can find disjoint $a_1,...,a_n\in\gC$ such that
$$
(\forall i)(\mu(a_i) = 1).
$$
The density arguments needed for $\scr{P}$ are summarised by the following lemma.
\begin{lem}\label{densityargs} Let $p\in \scr{P}$ and let $\gA$ be the domain of $p$. 
\begin{enumerate}[(D.1)]
\item There exists an exhaustive submeasure $\varphi:\scr{A}\rightarrow [0,1]\cap \scr{Q}$ extending $p$.\label{11}
\item For every $\epsilon\in (0,1]$ and antichain $(a_i)_{i\in\omega}$ from $\scr{A}$, there exists $q\in \scr{P}$ extending $p$, such that for some $i\in\omega$ we have $a_i\in \mathrm{dom}(q)$ and $q(a_i)< \epsilon$.\label{D2}
\item For any $n\in \scr{N}$ there exists an $n$-pathological submeasure $q\in \scr{P}$ such that $q\preceq p$.\label{D3}
\end{enumerate}
\end{lem}
\begin{proof} For (D.\ref{11}), let $a_0,...,a_n$ be the atoms of $\gA$ and for each $i\in n+1$ let $\varphi_i:\scr{A}_{a_i}\rightarrow \scr{Q}$ be a normalised finitely additive measure (take the Lebesgue measure for example). Let $\varphi$ be the submeasure promised by Lemma \ref{1611112}. To see that $\varphi$ is exhaustive let $(b_i)_{i\in\omega}$ be a disjoint sequence in $\scr{A}$ and fix $\epsilon >0$. For each $i\in n+1$ let $D_i = \{b_j\cap a_i:j\in\omega\}$. Then each $D_i$ is a disjoint sequence in $\scr{A}_{a_i}$ and so, since each $\varphi_i$ is exhaustive, we can find an $N$ such that
$$
(\forall m\geq N)(\varphi(b_m) = \varphi(\bigsqcup_{i\in n+1}b_m\cap a_i)\leq \sum_{i\in n+1}\varphi(b_m\cap a_i) = \sum_{i\in n+1}\varphi_i(b_m\cap a_i) \leq (n+1)\epsilon).
$$
Since $\epsilon$ was arbitrary (and $n$ was fixed) we are done.\\\\
The claim (D.\ref{D2}) follows from (D.\ref{11}) since we can find an exhaustive submeasure $\varphi:\scr{A}\rightarrow [0,1]\cap \scr{Q}$ extending $p$, so that for some $n$ we have $\varphi(a_n)<\epsilon$. Let $\gC$ be the algebra generated by $\gA\cup\{a_n\}$ and take $q= \varphi\restriction \gC$.\\\\
For (D.\ref{D3}), let $b_0,....,b_k$ be the atoms of $\gA$.  For each $i\in k+1$ let $b^i_1,...,b^i_n$ be a partition of $b_i$ into non-zero pieces. Let $\gC$ be the subalgebra of $\scr{A}$ generated by $\{b^i_j:i\in k+1,j\in [n]\})$. For $a\in \gC$ let 
\begin{equation}\label{0609121}
q(a) = p(\bigcap \{b\in \gC:a\leq b\}).
\end{equation}
Then $q:\gA\rightarrow [0,1]$ is a submeasure extending $p$. Also if, for $l\in [n]$, we let $a_l = \bigcup_{i\in k+1} b^i_l$ then (since $a_l\not\in \gA$ and $a_l$ intersects each atom of $\gC$) $q(a_l) = 1$ and of course the $a_l$ are pairwise disjoint.
\end{proof}
\begin{proof}[Proof of Theorem \ref{1611111}.] The fact that $\lambda\in\func{\gB}{\scr{R}}$ follows by the genericity of $G$ and the fact that for any $p\in \scr{P}$ and $a\not\in \mathrm{dom}(p)$, we can find $q\preceq p$ such that $a\in \mathrm{dom}(q)$ (for example, see (\ref{0609121})). It is a normalised submeasure because its restriction to any finite subalgebra of $\gB$ is. By (D.\ref{D3}), for each $n\in\scr{N}$, the set $\{p\in \scr{P}:\mbox{$p$ is $n$-pathological}\}$ is dense in $\scr{P}$ and so for each $n$, we can find an $n$-pathological $p\in G$. The disjoint sequence that witnesses this, $a_1,...,a_n$, is such that $(\forall i)(\lambda(a_i) = p(a_i)= 1)$. Thus $\lambda$ cannot be uniformly exhaustive. Suppose for a contradiction that for some antichain $(a_i)_{i\in\omega}$ in $M$ and $\epsilon \in (0,1]$ we have $(\forall i)(\lambda(a_i)\geq \epsilon)$. By (D.\ref{D2}) the set $D = \{p\in\scr{P}:(\exists i\in \scr{N})(a_i\in p \wedge p(a_i) < \epsilon)\}$ is dense. Thus we can find a $p\in  G\cap D$ and an $i\in\omega$ such that $a_i\in \mathrm{dom}(p) \subseteq \mathrm{dom}(\lambda)$ and $\lambda(a_i) = p(a_i) < \epsilon$, which is a contradiction.\end{proof}
\section{Talagrand's ideal}\label{forcingideal}
Recall that $\cT$ is the product space $\prod_{i\in\scr{N}}[2^i]$, $\scr{T} = \mathrm{Clopen}(\cT)$ and $\nu:\scr{T}\rightarrow \scr{R}$ is Talagrand's submeasure. We may extend $\nu$ to a $\sigma$-subadditive  submeasure on $\cP(\cT)$ by
\begin{equation}\label{100}
\nu(A) = \inf\{\sum_{i\in \scr{N}}\nu(A_i):A_i\in \scr{T}\wedge A\subseteq \bigcup_{i\in \scr{N}}A_i\}
\end{equation}
where the restriction of $\nu$ to $\mathrm{Borel}(\cT)$ is a continuous submeasure (see for example \cite[Proposition 7.1]{roberts91}). Plainly this extension remains pathological. Let 
$$
\pathh = \{A\in \cP(\cT):\nu(A) = 0\}.
$$
Let $\meag$ be the ideal of meagre subsets of $\cT$ and $\nulli$ be the ideal of Lebesgue null sets. For the rest of this section fix a countable transitive model $M$ of $\zfc$ and $\cU\in M$ such that, in $M$, $\cU$ is a non-principal ultrafilter. By $\nu^M$ we mean Talagrand's submeasure as defined in $M$ and with respect to $\cU$. By $\pathh_M$ we mean the collection (in $M$) of $\nu^M$-null sets. We will also denote the complete Boolean algebra $\mathrm{Borel}(\cT)/\pathh$, as computed in $M$, by $\pathh_M$. By $N$ we mean either a countable transitive model of $ZFC$ such that $M\subseteq N$ or $V$ itself. By $\nu^N$ we mean $\nu$ as defined in $N$ with respect to \emph{any} non-principal ultrafilter $\cV$ (in $N$) such that $\cU\subseteq \cV$. Such an ultrafilter exists since $\cU$ will always have the finite intersection property and will not contain any finite sets, and so any non-principal extension will do. We do not know if different ultrafilters produce different ideals, nevertheless, the choice of the ultrafilter here will not matter. We let $\pathh_N$ denote the collection of $\nu^N$-null sets. If $\cV$ is any non-principal ultrafilter over $\scr{N}$ we let $\nu_\cV$ be Talagrand's submeasure defined with respect to the ultrafilter $\cV$.\\\\
Given a subset $A$ of $\bc$ let $R(A) = \{f\in \cT:(\forall c\in A)(A_c\in \nulli\rightarrow f\not\in A_c)\}$. By $\bc$ we mean the collection of Borel codes as described in \cite[Chapter 25]{jech03}. If $H$ is a countable transitive model of $ZFC$ then $R(\bc\cap H)$ is just the collection of random reals over $H$. We prove the following.
\begin{thm}\label{mainnow} Let $G$ be a $\pathh_M$-generic filter over $M$. Then in $M[G]$ we have
$$
(\forall \cV)(\mbox{$(\cV$ is a non-principal ultrafilter on $\scr{N} \wedge \cU\subseteq \cV) \rightarrow \nu_\cV(R(\bc\cap M)) = 0$}).
$$
\end{thm}
The two claims we will need are as follows, the first is due to Christensen.
\begin{fact}\emph{(\cite[Theorem 1]{christ77})}\label{dualnull} There exists a Borel set $A$ such that $A\in \pathh$ and $\cT\setminus A \in \nulli$. \end{fact}
Fact \ref{dualnull} is saying that the ideals $\pathh$ and $\nulli$ are \emph{dual}, according to \cite{kunen84rand}.
\begin{prop}\label{1001124} If $c\in \bc\cap M$ then $\nu^M(A_c^M)\geq \nu^N(A_c^N)$. In particular for any $c\in \bc\cap M$, if $A_c\cap M \in\pathh_M$ then $A_c\cap N\in \pathh_N$.
\end{prop}
Assuming this for now, we have the following.
\begin{proof}[Proof of Theorem \ref{mainnow}] By Fact \ref{dualnull}, we can find $c,d\in \bc\cap M$ such that that $A_c\cap M\in \nulli_M$ and $A_d\cap M = \cT\setminus A_c\cap M\in \pathh_M$. Let $G$ be a $\pathh$-generic filter over $M$. In $M[G]$, if $f\in R(\bc\cap M)$ then $f\not \not\in A_c$ so that $R(\bc\cap M)\subseteq A_d$. But by Proposition \ref{1001124} we know that, since $A_d\cap M\in \pathh_M$, for any appropriate $\cV$ we have $\nu_\cV^{M[G]}(A_d\cap M[G]) = 0$. 
\end{proof}

\begin{proof}[Proof of Proposition \ref{1001124}] This proof follows the same induction argument as in \cite[Proposition N]{frem08}. Let $\cT^*$\index{$\cT^*$} be the collection $\bigcup_{I\in [\scr{N}]^{<\omega}}\prod_{n\in I}[2^n]$. Let $\phi_1(f,\tau)$ be the formula
$$
\tau\in \cT^*\wedge f\in \cT \wedge (\forall n\in \mathrm{dom}(\tau))(f(n)\neq \tau(n)))
$$
Of course 
$$
f\in \bigcap_{n\in \mathrm{dom}(\tau)}S_{n,\tau(n)} \leftrightarrow \phi_1(f,\tau).
$$
Since ${\cT^*}^M = \cT^*$ and $\cT^M = \cT\cap M$, we have
$$
(\forall \tau)(\forall f\in M)(\phi_1(f,\tau)\leftrightarrow \phi_1^M(f,\tau)).
$$
So if $\tau \in \cT^*$
\begin{equation}\label{2307125}
(\bigcap_{n\in I}S_{n,\tau(n)})^M = \{f:\phi_1(f,\tau)\}^M = \{f\in M: \phi_1^M(f,\tau)\} = \bigcap_{n\in I}S_{n,\tau(n)}\cap M.
\end{equation}
Let $\phi_2(x)$ be the formula
\begin{eqnarray*}
\mbox{$x$ is a function}\wedge \mathrm{dom}(x) = 3 &\wedge& (\exists \tau \in \cT^*)(\exists k)(x(0) = \bigcap_{n\in \mathrm{dom}(\tau)}S_{n,\tau(n)}\\&\wedge& x(1) = |\tau| \wedge x(2) = \left(\frac{\eta(k)}{|\tau|}\right)^{\alpha(k)}).
\end{eqnarray*}
Of course
$$
\phi_2(X)\leftrightarrow X\in \cD.
$$
By this and (\ref{2307125}) we see that 
\begin{equation}\label{2307126}
\cD^M = \{(A\cap M,I,w):(A,I,w)\in \cD\}.
\end{equation}
Note that the sequences $(\eta(k))_{k\in\scr{N}}$ and $(\alpha(k))_{k\in\scr{N}}$ are in $M$.\\\\
Now we proceed by downwards induction. Let $[k,p]$ be the statement that 
$$
(\cC_{k,p}^M = \{(A^M,I,w):(A^M,I,w)\in \cC_{k,p}\}) \wedge (\forall A\in \scr{T})(\psi_{\cC_{k,p}}^M(A^M) = \psi_{\cC_{k,p}}^N(A^N)).
$$
We show that for each $k\leq p$ the statement $[k,p]$ holds. First we show that $\psi_\cD^M(A^M) = \psi_\cD^N(A^N)$, this along with (\ref{2307126}) will prove $[p,p]$. Suppose $\psi_\cD^M(A^M)<\eta$, for some $\eta\in\scr{Q}_{>0}$. Then we can find $\{(X_i\cap M,I_i,w_i):i\in [N]\}\subseteq \cD^M$ such that $A^M = A\cap M\subseteq \bigcup_{i\in [N]}X_i\cap M = (\bigcup_{i\in [N]}X_i)^M$ and $\sum_{i\in [N]}w_i<\eta$. Thus $\{(X_i\cap N,I_i,w_i):i\in I\}\subseteq \cD^N$ witnesses $\psi_\cD^N(A)<\eta$. The other direction is the same but just using the fact that if $\{(X_i\cap N,I_i,w_i):i\in I\}\subseteq \cD^N$ then $\{(X_i\cap M,I_i,w_i):i\in I\}\subseteq \cD^M$.\\\\
Suppose now for some $[k+1,p]$ holds. By $[k+1,p]$, for every $s\in \cT^*$ and $B\in \scr{T}$, we have
$$
\psi_{k+1,p}^M((\pi^{-1}_{[s]}(B))^M) = \psi_{k+1,p}^N((\pi^{-1}_{[s]}(B))^N),
$$
from which it follows that
\begin{center}
$(\forall X \in \scr{T})($ $X\cap M$ is $(I,\psi_{k+1,p}^M)$-thin if and only if $X\cap N$ is $(I,\psi_{k+1,p}^N)$-thin).
\end{center}
From this, arguing as in the case for $[p,p]$, we obtain $[k,p]$.\\\\
Finally, since $\cU\subseteq \cV$, we have for each $k\in \scr{N}$:
\begin{center}
\emph{If $(X\cap M,I,w)\in \cE_{k}^M$ then $(X\cap N,I,w)\in \cE_{k}^N$,}
\end{center}
where of course $\cE_{k}^M = \{(X,I,w):\{p:(X,I,p)\in \cC_{k,p}^M\}\in\cU\}$ and $\cE_{k}^N = \{(X,I,w):\{p:(X,I,p)\in \cC_{k,p}^N\}\in\cV\}$. This completes the proof.
\end{proof}An ideal $\cI$ on $\mathrm{Borel}(\cT)$ is called \emph{analytic on $G_\delta$} if, for every $G_\delta$ set $A\subseteq \cT\times \cT$, the set $\{x:A_x\in \cI\}$ is analytic. By Proposition \ref{gdelta}, below, $\pathh$ is analytic on $G_\delta$. The following now follows from \cite{farahzap04}. \begin{thm}\label{groundrealz} $\pathh\force \check{\cT} \in \dot{\meag}\cap \dot{\nulli}$.
\end{thm}
The following should be compared to \cite[Theorem 17.25]{kech95}.
\begin{prop}\label{gdelta} Let $\mu:\mathrm{Borel}(\cT)\rightarrow \scr{R}$ be a Maharam submeasure. Then for each Borel set $A\subseteq \cT\times \cT$, the map
$$
x\mapsto \mu(A_x)
$$
is Borel. In particular, by considering the preimage of $\{0\}$, the ideal $\mathrm{Null}(\mu)$ is analytic on $G_\delta$.
\end{prop}
\begin{proof} Fix a Maharam submeasure $\mu$ on $\mathrm{Borel}(\cT)$. Given $A\in \mathrm{Borel}(\cT\times\cT)$, let $[A]$ be the statement:
$$
\mbox{The map $\cT\rightarrow \cT:x\mapsto \mu(A_x)$ is Borel}.
$$
We show that the collection of all $A\in \cP(\cT\times \cT)$ such that $[A]$ holds is closed under countable intersections of decreasing sequences and countable unions of increasing sequences and contains all open sets. By the Monotone Class Theorem (see \cite[Theorem 6B]{halmos50}) it follows that $[A]$ holds for each $A\in \mathrm{Borel}(\cT\times \cT)$. Indeed, let $(A_i)_{i\in\scr{N}}$ be a decreasing sequence such that $[A_i]$ holds for each $i$ and let $A = \bigcap_i A_i$. Let $f:\cT\rightarrow \scr{R}$ be the map $x\mapsto \mu(A_x)$ and, for each $n\in \scr{N}$, let $f_n:\cT\rightarrow \scr{R}$ be the map
$$
x\mapsto \mu((A_n)_x).
$$
By the monotonicity of $\mu$, we have that $f_1(x)\geq f_2(x)\geq \cdots$, and since $\mu$ is Maharam we have 
$$
f(x) = \lim_n f_n(x).
$$
Since each $f_n$ is Borel the map $f$ remains Borel (see \cite[Theorem 4.2.2]{dudley02}), and so we must have $[A]$. The same argument shows that, if $(A_i)_{i\in \scr{N}}$ is an increasing sequence such that $[A_i]$ holds for each $i$, then $[\bigcup_i A_i]$ also holds.
Let us now show that $[A]$ holds for each open set $A$ of $\cT\times \cT$.
If $A = \bigcup_{i\in [n]} [s_i]\times [t_i]\subseteq \cT\times \cT$, for some finite sequences $s_i$ and $t_i$, then for each $x\in \cT$ and function $\mu:\mathrm{Borel}(\cT)\rightarrow \scr{R}$ we have
$$
\mu(A_x) = \mu(\bigcup\{[t_i]: i\in [n]\wedge x\in [s_i]\}).
$$
From this it is straightforward to see that the map $x\mapsto \mu(A_x)$ is continuous (and so Borel). Now suppose $A$ is an open set in $\cT\times \cT$. Then we can find finite sequences $(s_i)_{i\in\scr{N}}$ and $(t_i)_{i\in \scr{N}}$ such that $A = \bigcup_{i\in\scr{N}}[s_i]\times [t_i]$. For each $n$, let $A_n = \bigcup_{i\in [n]} [s_i]\times [t_i]$. Then, by the above, we see that $[A]$ holds.
\end{proof}
Finally for this section let us remark that \cite{farahzap04} may be avoided in justifying Theorem \ref{groundrealz}. By \cite[Proposition N]{frem08}, the ideal $\pathh$ satisfies the following condition:
$$
(\forall f\in \cT)(\forall A\in \pathh)(A+f\in \pathh).
$$
This is called \emph{$0$-$1$-invariance} in \cite{kunen84rand}. Now we have the following.
\begin{fact}\emph{(\cite{kunen84rand})} If $\cI$ is a $0$-$1$-invariant ideal on $\cP$ that is dual to $\nulli$ then
$$
\mathrm{Borel}(\cT)/\cI\force \check{\cT}\in \dot{\nulli}.
$$
The above is still true if we replace $\nulli$ by $\meag$.\footnote{Note that the assumptions in this claim do not include the \emph{absoluteness} of $\cI$, as in \cite{kunen84rand}, but the proof still works.}
\end{fact}
Theorem \ref{groundrealz} now follows by Fact \ref{dualnull} and the following.
\begin{lem}\label{1001127} For every $A\in \mathrm{Borel}(\cT)\setminus \pathh$ there exists $B\in (\mathrm{Borel}(\cT)\cap \meag)\setminus \pathh$ such that $B\subseteq A$. In particular, $\meag$ and $\pathh$ are dual.
\end{lem}
\begin{proof} Suppose that for some $A\in \mathrm{Borel}(\cT)\setminus\pathh$ we have $\mathrm{Borel}(\cT)\cap\meag\cap \cP(A)\subseteq \pathh$. Let $\dot{r}$ be a name such that 
\begin{equation}\label{1101121}
\pathh\force \mbox{$(\forall c\in \check{\bc})(A_c\cap \check{\cT}\in \check{\pathh}\rightarrow \dot{r}\not\in {A_c})$}.\footnote{See \cite[Proposition 2.1.2]{zap08}.}
\end{equation}
We claim that 
\begin{equation}\label{2307127}
A\force \mbox{`$\dot{r}$ is a Cohen real'}.
\end{equation}
If not then for some $B\subseteq A$ and some $c\in \bc$ with $A_c\in \meag$ we have $B\force \mbox{$\dot{r}\in A_c$}$. If $d\in \bc$ is such that $B = A_d$ then $B\force \mbox{$\dot{r}\in A_d\cap A_c$}$. Let $e\in \bc$ be such that $A_e = A_c\cap A_d$. But then 
$$
A_c\cap A_d\in \mathrm{Borel}(\cT)\cap \meag\cap \cP(A)\subseteq \pathh.
$$ In particular $B\force \mbox{$A_e\cap \check{T}\in \check{\pathh} \wedge \dot{r}\in A_e$}$, which contradicts (\ref{1101121}). Thus (\ref{2307127}) holds which contradicts the fact that a Maharam algebra cannot add a Cohen real.\\\\
Now use the above to find, for each $A\in \mathrm{Borel}(\cT)\setminus \pathh$, a meagre Borel set $\Gamma(A)\not\in \pathh$ such that $\Gamma(A)\subseteq A$. Let $B_1 = \Gamma(\cT)$. If $B_\beta$ for $\beta<\alpha<\omega_1$ has been constructed let
$$
B_{\alpha} = \left\{\begin{array}{cl} \Gamma(\cT\setminus (\bigcup_{\beta<\alpha}B_\beta)),\quad&\mbox{if $\cT\setminus (\bigcup_{\beta<\alpha}B_\beta)\not\in \pathh,$}\\
\emptyset,\quad&\mbox{otherwise.}\end{array}\right.
$$
Since $\mathrm{Borel}(\cT)/\pathh$ is ccc, we know that $B := \{B_\alpha:\alpha<\omega_1 \wedge B_\alpha\not\in \pathh\}$ is countable. Thus $\cT\setminus \bigcup B \in \pathh$ and $\bigcup B\in \meag$, since each $B_\alpha\in\meag$.\end{proof}
\section{Talagrand's $\psi$}\label{calcs}
Let $\lambda$ be the Lebesgue measure on $\mathrm{Clopen}(2^\omega)$. If $A\in \mathrm{Clopen}(2^\omega)$ then we know that for some $n\in\omega$ we have
$$
\lambda(A) = |\{s\in \func{n}{2}:[s]\subseteq A\}|\cdot 2^{-n}.
$$
A similarly explicit description for Talagrand's submeasure would be very useful. To this end, we investigate the first (pathological) submeasure constructed in \cite{tal08}, this is the submeasure denoted by $\psi$ in Section \ref{prels} (and in \cite{tal08}). We remark that in \cite{tal08} the value $\eta(k)$ was set to $2^{2k+10}2^{(k+5)^4}(2^3+2^{k+5}2^{(k+4)^4})$. As pointed out in \cite{tal08} anything larger will do, so for simplicity we take the value \index{$\eta(k)$|}
$$
\eta(k) = 2^{2500k^4}.
$$ 
In Subsection \ref{calcs1} we prove the following:
\begin{thm}\label{thecalcs} We have $\psi(\cT) = \eta(1)^{\alpha(1)} = 2^{\frac{2500}{216}}$ and 
$$
\psi([s]) =  \min\{2^{-\delta(|I|)+1},2^{-\delta(|I|)}\left(\frac{\eta(\delta(|I|))}{|I|}\right)^{\alpha(\delta(|I|))}\},
$$
where $s\in \prod_{i\in I}[2^i]$, for some finite $I\subseteq\scr{N}$, and $\delta(m) = \min\{n\in\scr{N}:\eta(n)\geq m\}$.
\end{thm}
In Subsection \ref{r7} we list the numerical inequalities that we shall need for Subsection \ref{calcs1}.
\subsection{Main calculations}\label{calcs1}
We begin with the following two definitions.
\begin{df} For $X\in\scr{T}$ we say that $X$ is a \eb{$\cD$-set}\index{D-set@$\cD$-set} if and only if for some (non-empty) finite set $I\subseteq \scr{N}$ and some $\tau\in \prod_{n\in I}[2^n]$ we have 
$$
X = \bigcap_{n\in I}\{y \in \cT:(\forall n\in I)(y(n)\neq \tau(n)\}) = \bigcap_{n\in I} S_{n,\tau(n)}.
$$
Since we can recover $I$ and $\tau$ from $X$ we allow ourselves to denote $I$ by $X^\ind$ and $\tau(n)$ by $X(n)$\index{$X^\ind$}\index{$X(n)$}. 
\end{df}
\begin{df} Let $A\subseteq \cT$, $X$ a collection of $\cD$-sets and $Y \in [\cD]^{<\omega}$. We say that $X$ (resp. $Y$) is a \bld{\emph{cover of $A$}} if and only if $A\subseteq \bigcup X$ (resp. $A\subseteq \bigcup Y$). We say that $X$ (resp. $Y$) is a \emph{\bld{proper cover of $A$}} if and only if it is a cover of $A$ and for any $X'\subsetneq X$ (resp. $Y'\subsetneq Y$)
$$
A\not\subseteq \bigcup X' \mbox{(resp. $A\not\subseteq \bigcup Y'$)}.
$$
\end{df}
Clearly then given $A\subseteq \cT$ we have
\begin{equation}\label{r7}
\psi(A) = \inf\{w(X):\mbox{$X\subseteq \cD$ and $X$ properly covers $A$}\}.
\end{equation}
Let us now measure $\cT$. The idea here is as follows. For each proper cover $X$ of $\cT$ we find another cover $Y$ of $\cT$ of lower weight, where the $Y$ here will have a very regular structure and so will have an easily calculable weight. Of course it will be sufficient to consider the infimum over all such regular structures. 
\begin{df}\label{34}
For any $n\in \scr{N}$ let
$$
w(n) = 2^{-\delta(n)}\left(\frac{\eta(\delta(n))}{n}\right)^{\alpha(\delta(n))}.
$$
If $X$ is a finite collection of $\cD$-sets then we will denote the \eb{weight}\index{weight of a $\cD$-set} of $X$ by
$$
w(X) = \sum_{Y\in X}w(|Y^\ind|).\footnote{We now are using the term \emph{weight} for $\cD$-sets and members of $\cD$, but with a slight variation in meaning.}
$$
\end{df}
By Inequality \ref{ineq1} from Subsection \ref{r7}, we see that if $X$ is a $\cD$-set then $w(|I(X)|)$ will be the least weight that we can possibly attach to it. Specifically, we will always have $(X,I(X),w(|I(X)|))\in \cD$ and, if $(X,I(X),w)\in\cD$ then $w\geq w(|I(X)|)$.\\\\
Here is the regular structure we mentioned above.
\begin{df}\label{r24} Let $X = \{X_i :i\in I\}$ be a collection of $\cD$-sets. We call $X$ an \emph{\bld{$N$-rectangle}}\index{rectangle@$N$-rectangle} for some integer $N\geq 2$ if and only if the following hold:
\begin{itemize}
\item $|I|= N$;  
\item $X_i^\ind = X_j^\ind$ for all $i,j\in I$;
\item $X_i(m) \neq X_j(m)$, whenever $i\neq j$ and $m\in X_i^\ind$;
\item $|X_i^\ind| = N-1$ for all (any) $i\in I$.
\end{itemize}
\end{df}
Notice that the weight of an $N$-rectangle is given by
\begin{equation}\label{rectanglesize}
N\cdot w(N-1).
\end{equation}
Rectangles give rise to proper covers of $\cT$:
\begin{lem}\label{17} If $X := \{X_i:i\in I\}$ is an $N$-rectangle then $X$ is a proper cover of $\cT$.
\end{lem}
\begin{proof} Assume that $x\in \cT\setminus \bigcup_i X_i$. Then for each $i$ we can find an $m_i\in X_i^\ind$ such that $x(m_i) = X_i(m_i)$. These $m_i$ must be distinct for if $i\neq j$ and $m:=m_i = m_j$, then $X_i(m) = x(m) = X_j(m)$, for some $i$, contradicting the third item from Definition \ref{r24}. But then $\{m_1,...,m_N\}\subseteq X_i^\ind$ a (cardinality) contradiction. To see that this cover is proper let $J$ be a non-empty strict subset of $\{1,2,...,N\}$. Then $|J|\leq N-1 = |X_i^\ind|$, for each $i\in J$. Enumerate 
$$
J = \{a_1,a_2,...,a_k\}.
$$
Inductively, choose $b_1\in X_{a_1}^\ind$, $b_2\in X_{a_2}^\ind\setminus\{b_2\}$, $b_3\in X_{a_3}^\ind\setminus\{b_1,b_2\},...,b_k\in X_{a_k}^\ind\setminus\{b_1,...,b_{k-1}\}$. Now define $y\in \prod_{i\in J}[2^{b_i}]$ by
$$
y_{i} = \left\{\begin{array}{cl} X_{a_i}(b_i),\quad&\mbox{if $i\in \{b_1,...,b_k\}$;}\\
1,\quad&\mbox{if $i\not\in J$.}\end{array}\right.
$$
and note that $y\in \cT\setminus \bigcup_{i\in J}X_i$.\end{proof}
Given a proper cover of $\cT$ we claim that we can find an $N$-rectangle of lower weight. 
Before we can demonstrate this we need one more claim.
\begin{lem}\label{30} Let $X = \{X_i:i\in I\}$ be a collection of $\cD$-sets that properly covers $\cT$. Then
$$
|\bigcup_{i\in I}X_i^\ind|\leq |I|-1.
$$
\end{lem}
\begin{proof} For each $i\in I$ let $I_i = X_i^\ind$. Recall that a \emph{complete system of distinct representatives} for $\{I_i:i\in I\}$ (a CDR) is an injective function $F:I\rightarrow \bigcup_{i\in I}I_i$ such that $(\forall i\in I)(F(i)\in I_i)$, and that by Hall's marriage theorem a CDR exists if and only if
$$
(\forall J\subseteq I)(|J|\leq |\bigcup_{i\in J}I_i|).\footnote{See \cite{hall35}.}
$$
Clearly if a CDR existed for $\{I_i:i\in I\}$ then $\bigcup_{i\in I}X_i$ would not cover $\cT$ (just argue as in the proof of Lemma \ref{17}). So for some $J\subseteq I$ we have $|\bigcup_{i\in J}I_i|\leq |J|-1$. Assume that $|J|$ is as large as possible so that
\begin{equation}\label{1311111}
(J'\subseteq I \wedge |J'|>|J|) \rightarrow (|J'|\leq |\bigcup_{i\in J'}I_i|).
\end{equation}
If $J = I$ then we are done. So we may assume that $J\subsetneq I$. Since $X$ is a proper cover of $\cT$ there exists $t\in \cT$ such that $t\not\in \bigcup_{i\in J}X_i$. For $i\in I\setminus J$ let $I'_i =I_i\setminus \bigcup_{j\in J}I_i$. Suppose that $\{I'_i:i\in I\setminus J\}$ has a CDR $F:I\setminus J\rightarrow \bigcup_{i\in I\setminus J} I'_i$. Let $s \in \prod_{k\in\mathrm{ran}(F)}[2^k]$ be defined by $s(k) = X_{F^{-1}(k)}(k)$. Then the function $(t\setminus \{(k,t(k)):k\in \mathrm{ran}(F)\})\cup s\not\in \bigcup_{i\in I} X_i$, which is a contradiction. Thus no such CDR can exist and so by Hall's theorem again, we may find a $J'\subseteq I\setminus J$ such that $|\bigcup_{i\in J'}I'_i|\leq |J'|-1$. But then 
$$
|\bigcup_{i\in J\cup J'}I_i| = |\bigcup_{i\in J}I_i \cup \bigcup_{i\in J'} I'_i| \leq |J|-1+|J'| -1\leq |J|+|J'|-1 = |J\cup J'|-1.
$$
But $|J\cup J'|>|J|$, contradicting (\ref{1311111}).
\end{proof}
\begin{prop} For every proper cover of $\cT$ there exists an $N$-rectangle of lower weight.
\end{prop}
\begin{proof} Let $X = \{(X,I_i,w_i):i\in [M]\}$ is a proper cover of $\cT$  and assume that $I_1$ is such that $(\forall i\in [M])(w(|I_1|) \leq w(|I_i|))$. By Lemma \ref{30} we have
\begin{equation}\label{2706121}
(\forall i)(|I_i|+1\leq |\bigcup_{i\in [N]}I_i| + 1\leq M).
\end{equation}
So if $Y$ is an $|I_1| + 1$-rectangle we get:
$$
w(X) \geq \sum_{i\in [M]}w(|I_i|)\geq M w(|I_1|) \geq (|I_1|+1)w(|I_1|) \overset{(\ref{rectanglesize})}{=} w(Y).
$$
\end{proof}
Thus we have
\begin{equation}\label{r12}
\psi(\cT) = \inf\{w(X):\mbox{$X$ is an $N$-rectangle, for some $N$}\}.
\end{equation}
But by Inequality \ref{ineq4} we see that $\psi(\cT)$ is just the weight of a $2$-rectangle, that is to say,
\begin{equation}\label{entirespacemeas}
\psi(\cT) = \eta(1)^{\alpha(1)}.
\end{equation}
This gives the first half of Theorem \ref{thecalcs}.\\\\
Now let us try to measure sets of the form $[s]$. Fix a non-empty finite subset $\cI$ of $\scr{N}$ and an $\tau\in \prod_{i\in \cI}[2^i]$ and lets measure $A := [\tau]$.\\\\
Note that as before
$$
\psi(A) = \inf\{w(X): \mbox{$X\subseteq \cD$ is a proper cover of $A$}\}.
$$
The idea here is the same as before but instead of rectangles we use the following analogue of Definition \ref{r24} and also Definition \ref{2506126}, below. 
\begin{df}\label{2506124} Let $X := \{X_i:i\in I\}$ be a collection of $\cD$-sets. We call $X$ a \emph{\bld{$(J,S,N)$-rectangle}}\index{rectangle@$(J,S,N)$-rectangle} for some non-empty finite subset $J$ of $\scr{N}$, $S\subseteq \prod_{j\in J}[2^j]$ and integer $N\geq 2$ if and only if the following hold:
\begin{itemize}
\item $J\subsetneq X_i^\ind$, always;
\item $\{\bigcap_{l\in X_i^\ind\setminus J}S_{l,X_i(l)}:i\in I\}$ is an $N$-rectangle;
\item $(\forall s\in S)(\forall i\in I)(\forall j\in J)(X_i(j) \neq s(j))$.
\end{itemize}
In the case that $S= \{s\}$, we shall call $X$ a \eb{$(J,s,N)$-rectangle}\index{rectangle@$(J,s,N)$-rectangle}.
\end{df}
For example, in the case that $J = [m]$, for some $m\in \scr{N}$, this new type of rectangle is just an old rectangle with $m$ rows attached to the bottom (most likely with a gap) where the values of the determining sequences along these rows miss the corresponding values of $s$.\\\\
Of course the weight of a $(J,s,N)$-rectangle is given by
$$
N\cdot w(|J|+N-1)
$$
\begin{lem} Every $(\cI,\tau,N)$-rectangle covers $A$.
\end{lem}
\begin{proof} Let $X = \{X_i:i\in I\}$ be an $(\cI,\tau,N)$-rectangle, as in the above statement. Assume that we can find a $y\in A\setminus \bigcup X$. The assumption that $y\not\in X$ cannot be witnessed by $y(i)$ for some $i\in \cI$ since for each such $i$, we have $y(i) = \tau(i)\neq X_j(i)$, for each $j\in I$. In particular $y$ witnesses that $Y = \{\bigcap_{l\in X_i^\ind\setminus \cI}S_{l,X_i(l)}:i\in I\}$ does not cover $\cT$, which contradicts Lemma \ref{17} and the fact that $Y$ is an $N$-rectangle.
\end{proof}
Next we see how to use Lemma \ref{30} in this new situation and adapt what we have already done with $\psi(\cT)$ to $\psi(A)$ (compare (\ref{2706121}) above, and (\ref{2506123}) below).
\begin{lem}\label{43} If $X = \{X_i:i\in I\}$ is a proper cover of $A$ such that $(\forall i\in \cI)(X_i^\ind\setminus \cI\neq 0)$ then $\{\bigcap_{l\in X_i^\ind\setminus \cI}S_{l,X_i(l)}:i\in I\}$ is a proper cover of $\cT$
\end{lem}
\begin{proof} 
For each $i\in I$, let $I_i = X^\ind_i$. Let $X'_i = \bigcap_{l\in I_i\setminus \cI}S_{l,X_i(l)}$ and lets show that $Y := \{X_i':i\in I\}$ is a cover of $\cT$.
Suppose not and let $x\in \cT\setminus \bigcup Y$. Thus for every $i\in I$ there exists an $m_i\in X_i^\ind\setminus \cI$ such that $x(m_i) = X_i(m_i)$. Let $y\in A$ be such that $y(j) = x(j)$, for each $j\in \{m_i:i\in I\}$. Then it is straightforward to see that $y\not\in \bigcup X$, which contradicts the assumption that $X$ is a cover of $A$. Suppose now that $Y$ is not proper. Then there exists $I'\subsetneq I$ such that $\{X_i':i\in I'\}$ is a cover of $\cT$. But then $\{X_i:i\in I'\}$ is a cover of $A$, contradicting the properness of $X$.
\end{proof}
\begin{df}\label{2506126} A $\cD$-set $X$ is a \emph{\bld{$(I,S,J)$-spike}}\index{spike@$(I,S,J)$-spike} for some non-empty finite subset $I$ of $\scr{N}$, $S\subseteq \prod_{j\in I}[2^j]$ and $J\subseteq I$ if and only if $X$ is of the form 
\begin{equation}\label{47}
X = \bigcap_{j\in J} S_{j,t(j)}
\end{equation} 
such that $t\in \prod_{j\in J}[2^j]$ and $(\forall s\in S)(\forall j\in J)(t(j)\neq s(j))$. In the case that $S = \{s\}$, we shall call $X$ an \eb{$(I,s,J)$-spike}\index{spike@$(I,s,J)$-spike}.
\end{df}
Of course, every $(\cI,\tau,J)$-spike covers $A$.
\begin{prop}\label{1211111} For every proper cover of $A$ there exists an $(\cI,\tau,J)$-spike of lower weight.
\end{prop}
Assuming this for now we obtain
\begin{eqnarray}
\psi(A) &=& \min\{w(X): \mbox{$X$ is an $(\cI,\tau,J)$-spike for some $J\subseteq \cI$}\}\nonumber\\
 &=& \min\{2^{-\delta(|\cI|)+1},w(|\cI|)\}.\label{atommeas}
\end{eqnarray}
which completes the proof of Theorem \ref{thecalcs}.
\begin{proof}[Proof of Lemma \ref{1211111}]Let $X = \{(X_i,I_i,w_i):i\in [N]\}$ be a proper cover of $A$ and let $m = |\cI|$. If there exists $i\in [N]$ such that $|I_i|\leq m$ then any $(\cI,\tau,J)$-spike such that $|J| = |I_i|$ will have a lower weight than $X$ and will cover $A$ and we will be done. So we may assume that
\begin{equation}\label{2506122}
(\forall i\in [N])(|I_i|> m).
\end{equation}
By Lemma \ref{43} and Lemma \ref{30} we get
\begin{equation}\label{2506123}
(\forall i\in [N])(|I_i| \leq N+m-1).
\end{equation}
We now divide the proof into the following cases.
\begin{itemize}
\item $\delta(N+m-1) = 1$. Then
$$
w(X) = \sum_{i\in [N]}2^{-1}\left(\frac{\eta(1)}{|I_i|}\right)^{\alpha(1)} \overset{(\ref{2506123})}{\geq} N 2^{-1}\left(\frac{\eta(1)}{N+m-1}\right)^{\alpha(1)},
$$
and this lower bound can be achieved by any $(\cI,\tau,N)$-rectangle.
\item $\delta(N+m-1) >1$. Let $\delta_1 = \delta(N+m-1)-1$, $\delta_2 = \delta(N+m-1)$, $J_{1} = \{i\in [N]:\delta(|I_i|)\leq \delta_1\}$ and $J_2 = [N]\setminus J_1$. Of course
\begin{equation}\label{2506121}
\eta(\delta_1)<N+m-1\leq \eta(\delta_2).
\end{equation}
Notice that if $2>\eta(\delta_1) -m + 1$ then
$$
(\forall i\in [N])(\eta(\delta_1)\overset{(\ref{2506121})}{\leq} m \overset{(\ref{2506122})}{<} |I_i|\leq N+m-1 \leq \eta(\delta_2)),
$$
and so
$$
w(X) = \sum_{i\in [N]}2^{-\delta_2}\left(\frac{\eta(\delta_2)}{|I_i|}\right)^{\alpha(\delta_2)} \overset{(\ref{2506123})}{\geq} N2^{-\delta_2}\left(\frac{\eta(\delta_2)}{N+m-1}\right)^{\alpha(\delta_2)},
$$
which can be achieved by any $(\cI,\tau,N)$-rectangle. So we may assume that 
\begin{equation*}
2\leq \eta(\delta_1) -m + 1.
\end{equation*} By Inequality \ref{ineq6} we have
$$
w(X) = \sum_{i\in J_1} w_{|I_i|} + \sum_{i\in J_2} w_{|I_i|} \geq |J_1|2^{-\delta_1} + |J_2|2^{-\delta_2}\left(\frac{\eta(\delta_2)}{N+m-1}\right)^{\alpha(\delta_2)}.
$$
\begin{itemize}
\item $2^{-\delta_2}\left(\frac{\eta(\delta_2)}{N+m-1}\right)^{\alpha(\delta_2)}\leq 2^{-\delta_1}$. Then 
$$
w(X) \geq N2^{-\delta_2}\left(\frac{\eta(\delta_2)}{N+m-1}\right)^{\alpha(\delta_2)},
$$
which can be achieved by any $(\cI,\tau,N)$-rectangle..
\item $2^{-\delta_2}\left(\frac{\eta(\delta_2)}{N+m-1}\right)^{\alpha(\delta_2)} > 2^{-\delta_1}$. Then
$$
w(X) \geq N2^{-\delta_1}\overset{(\ref{2506121})}{>}(\eta(\delta_1)-m+1)2^{-\delta_1}.
$$
But this can be achieved by any $(\cI,\tau,\eta(\delta_1) -m + 1)$-rectangle since
$$
w(\eta(\delta_1)-m+1) = 2^{-\delta(\eta(\delta_1))}\left(\frac{\eta(\delta(\eta(\delta_1)))}{\eta(\delta_1)}\right)^{-\alpha(\delta(\eta(\delta_1)))} = 2^{-\delta_1}\left(\frac{\eta(\delta_1)}{\eta(\delta_1)}\right)^{-\alpha(\delta_1)} =  2^{-\delta_1}.
$$
\end{itemize}
\end{itemize}
Now, by Inequality \ref{ineq5}, any $(\cI,\tau,\cI)$-spike has a lower weight than any $(\cI,\tau,k)$-rectangle, and this completes the proof.
\end{proof} 
\subsection{Inequalities}\label{r7}
Here we list the inequalities that are needed for Subsection \ref{calcs1}. 
\begin{ineq}\label{ineq1} For each $k\in\scr{N}$ and $n\in [\eta(k)]$
$$
2^{-k}\left(\frac{\eta(k)}{n}\right)^{\alpha(k)}<2^{-(k+1)}\left(\frac{\eta(k+1)}{n}\right)^{\alpha(k+1)}.
$$
\end{ineq}
\begin{ineq}\label{ineq6} For $\delta_1,\delta_2,k\in\scr{N}$ such that $\delta_1\leq \delta_2 $ and $k \in [\eta(\delta_1)]$ we have
$$
2^{-\delta_1}\left(\frac{\eta(\delta_1)}{k}\right)^{\alpha(\delta_1)}\geq 2^{-\delta_2}.
$$
\end{ineq}
\begin{ineq}\label{ineq4} Let $N,M,\delta_1,\delta_2\in \scr{N}$ be such that $2\leq N \leq M$ and $\delta_1\leq \delta_2$. 
Then
$$
N2^{-\delta_1}\left(\frac{\eta(\delta_1)}{N-1}\right)^{\alpha(\delta_1)}\leq M2^{-\delta_2}\left(\frac{\eta(\delta_2)}{M-1}\right)^{\alpha(\delta_2)}.
$$ 
\end{ineq}
\begin{ineq}\label{ineq5} Let $k,N,\delta_1,\delta_2\in\scr{N}$ be such that 
$\delta_1\leq \delta_2$. Then
\begin{equation}\label{2310101}
2^{-\delta_1}\left(\frac{\eta(\delta_1)}{k}\right)^{\alpha(\delta_1)}\leq N2^{-\delta_2}\left(\frac{\eta(\delta_2)}{N+k-1}\right)^{\alpha(\delta_2)}.
\end{equation}
\end{ineq}
\section{Submeasures and signed measures}\label{measures}
We begin with the following definition.
\begin{df}\label{0909122} If $\gB$ is a Boolean algebra, call a collection $\{a_i:i\in [n]\}\subseteq \gB$, \eb{$*$-free}\index{free@$*$-free} if and only if for every non-empty $J\subseteq [n]$ we have
$$
\left(\bigcap_{j\in J} a_j\right)\cap\left(\bigcap_{j\not\in J}a^c_{j}\right)\neq 0 \wedge \bigcup_{i\in [n]}a_i = {1}.
$$ 
\end{df}
In this section we prove the following.
\begin{thm}\label{2911114} For every countable Boolean algebra $\gA$ there exists a Boolean algebra $\gB$ and an injective map $\gf:\gA\rightarrow \gB$ with the following properties:
\begin{enumerate}[\hspace{0.0cm}(T.1)]
\item $\gB = \La \gf[\gA] \Ra$, in particular $\gB$ will also be countable;\label{TT1}
\item if $\gA'\subseteq \gA$ is a finite subalgebra, then the collection $\gf[\mathrm{atoms}(\gA')]$ is $*$-free in $\gB$;\label{TT2}
\item $(\forall a,b\in \gA)(\gf(a\cup b) = \gf(a)\cup \gf(b))$.\label{TT3}
\end{enumerate}
Moreover, if $\g{D}$ is a Boolean algebra and $\g{g}:\gA\rightarrow \g{D}$ satisfies the above, then for any functional $\mu$ on $\gA$, there exists a unique signed finitely additive measure $\lambda$ on $\g{D}$ such that $\mu(a) = \lambda(\g{g}(a))$, for each $a\in \gA$. 
\end{thm}
Thus to each functional on a given countable Boolean algebra we associate a signed measure. In fact this association will be a linear map from the real vector space of all functionals on $\gA$ to the real vector space of all signed measures on $\gB$. We are of course interested in the case when $\mu$ is a submeasure. Unfortunately even for very simple submeasures, the corresponding measure may be unbounded.\\\\
At the end of this section we give an explicit example of such an $\gf$. More specifaclly, given a sequence $(X_i)_{i\in\scr{N}}$ of finite non-empty sets, we construct another sequence $(Y_i)_{i\in\scr{N}}$, consisting also of finite non-empty sets, and an injective map 
$$
\g{f}:\mathrm{Clopen}(\prod_{i\in\scr{N}}X_i)\rightarrow \mathrm{Clopen}(\prod_{i\in\scr{N}}Y_i)
$$ 
that satisfies properties (T.2) and (T.3) of Theorem \ref{2911114}. To obtain the rest of Theorem \ref{2911114}, we can take $\gB = \La \g{f}[\gA] \Ra$ where $\gA:= \mathrm{Clopen}(\prod_{i\in\scr{N}}X_i)$.\\\\
To illustrate the motivating idea behind Theorem \ref{2911114} consider the submeasure $\mu$ defined on the finite Boolean algebra $\gA$ of two atoms, $a$ and $b$, given by
$$
\mu(a) = \mu(b) = \frac{3}{4}, \quad \mu(a\cup b) = 1.
$$
This is clearly not additive. If we supposed for a moment that $\mu$ was additive then $a$ and $b$ would have to intersect. Thus we view the atoms $a$ and $b$ as not having enough \emph{space} for the submeasure $\mu$. We try to insert this space by allowing $a$ and $b$ to intersect, and by doing so we turn $\mu$ into a measure. To this end we consider the algebra $\gB$ of three atoms $c,d$ and $e$ and the map $\gf:\gA\rightarrow\gB$ define by
$$
a\mapsto c\cup d, b\mapsto d\cup e, a\cup b\mapsto c\cup d\cup e.
$$
The atom $d$ then becomes the inserted space, and on $\gB$ we can take the measure 
$$
\lambda(c) = \lambda(e) = \frac{1}{4}, \quad \lambda(d) = \frac{1}{2}
$$
Notice that no matter what values we had for $\mu$, we would still be able to solve (uniquely) for $\lambda$ and so we have a finite version of Theorem \ref{2911114}. Indeed, one need only solve the following system of linear equations:
$$
\lambda(c)+\lambda(d) = \mu(a), \quad \lambda(d)+\lambda(e) = \mu(b), \quad \lambda(c)+\lambda(d)+\lambda(e) = 1.
$$
The final $\gf$ and $\gB$ will be obtained as a direct limit of these finite constructions.\\\\
In this way we obtain the definition of $*$-free from Definition \ref{0909122}, and the following: 
\begin{df} For $n\in\scr{N}$ let $\sfr{n}$ be the Boolean algebra $\cP(\cP([n])^+)$. Call the sets $\{y\in \cP([n])^+:i\in y\}$, for $i\in [n]$, the \eb{$*$-free  generators} of $\sfr{n}$\index{free generators@$*$-free generators}.
\end{df}
\begin{rem} Clearly the $*$-free generators of $\sfr{n}$ are $*$-free and generate $\sfr{n}$. If $\fr{n}$ is the freely generated Boolean algebra over $n$ elements with free generators $a_1,...,a_n$ then $\sfr{n}$ may be viewed as the Boolean algebra
$$
\fr{n}_{\bigcup_{i=1}^na_i} (= \{a\in \fr{n}:a\subseteq \bigcup_{i\in [n]}a_i\}).
$$
\end{rem}
In the motivating example we see that the algebra $\gB$ with three atoms $c,d$ and $e$ is given by $\sfr{2}$ where we can take
$$
c = \{\{1\}\}, d = \{\{1,2\}\}, e = \{\{2\}\}.
$$
Notice that the atoms of $\gA$ are mapped to the $*$-free generators of $\gB$ ($=\sfr{2}$).\\\\ 
The fact that we can always solve for $\lambda$ (as in the motivating example) is given by the following two lemmas.
\begin{lem}\label{0612111} For each $n\in\scr{N}$ enumerate $\cP([n])^+ = \{y_i:i\in [2^{n}-1]\}$. Then the matrix $(a_{ij})_{i,j\in [2^n-1]}$ defined by
$$
a_{ij} = \left\{\begin{array}{cl} 1,\quad&\mbox{if $y_i\cap y_j\neq 0$;}\\
0,\quad&\mbox{otherwise.}\end{array}\right.
$$
is invertible.
\end{lem}
Since we could not find a particularly enlightening proof of Lemma \ref{0612111} we leave it to the end of this section.
\begin{lem}\label{9712111} Let $a_1,...,a_n$ be the $*$-free generators of $\sfr{n}$ and $\mu:\{\bigcup_{i\in I}a_i:I\in \cP([n])^+\}\rightarrow \scr{R}$ any functional. Then there exists a unique signed measure $\lambda:\sfr{n}\rightarrow \scr{R}$ such that $(\forall I\in \cP([n])^+)(\lambda(\bigcup_{i\in I}a_i) = \mu(\bigcup_{i\in I}a_i))$. 
\end{lem}
\begin{proof} Since we only need to decide the values that $\lambda$ should take on the atoms of $\sfr{n}$ we need only find a solution to the following set of linear equations:
$$
\sum_{y\in \cP([n])^+ \wedge q\cap y \neq 0} X_y = \mu(\bigcup_{i\in q} a_i) : q\in \cP([n])^+.
$$
These equations have a unique solution by Lemma \ref{0612111}. Now set $\lambda(\{y\}) = X_y$.
\end{proof}
\begin{df} Let $\gA$ be a finite Boolean algebra with $n$ atoms. A map $f:\gA\rightarrow \sfr{n}$ is called \index{good map@$\gA$-good map}$\gA$-\eb{good} if and only if the following hold:
\begin{itemize}
\item $f$ injectively maps the atoms of $\gA$ onto the $*$-free generators of $\sfr{n}$;
\item for each $a\in \gA$ we have $f(a) = \bigcup\{f(b):b\in\mathrm{atoms}(\gA)\wedge b\leq a\}$.
\end{itemize}
\end{df}
Of course in the context of the above definition any map sending the atoms of $\gA$ onto the $*$-free generators of $\sfr{n}$, induces an $\gA$-good map (by just taking unions).
\begin{lem}\label{2911111} Let $\gA$ be a finite Boolean algebra with $n$ atoms and let $f$ be an $\gA$-good map. Then $f$ is injective and satisfies the following properties:
\begin{itemize}
 \item $f(0) = 0_{\sfr{n}}$ and $f({1}) = {1}_{\sfr{n}}$,
 \item $(\forall a,b\in \gA)(f(a\cup b) = f(a)\cup f(b))$.
\end{itemize}
Moreover, for any functional $\mu$ on $\gA$, we can find a unique signed measure $\lambda$ on $\sfr{n}$ such that $(\forall a\in \gA)(\mu(a) = \lambda(f(a)))$.
\end{lem}
\begin{proof} The properties of $f$ follow by definition. The last part is just Lemma \ref{9712111}.\end{proof}
The fact that we can coherently put together  the maps from Lemma \ref{2911111} to build the map $\gf$ from Theorem \ref{2911114} is justified by following two lemmas.
\begin{lem}\label{2911112} Let $n\in \scr{N}$ and for each $i\in [n]$, let $a_i = \{y\in\cP([n])^+:i\in y\}$, so that $a_1,...,a_n$ are the $*$-free generators of $\sfr{n}$. Let $\gB$ be a finite Boolean algebra and let $b_1,...,b_n$ be $*$-free members of $\gB$. Then the map $a_i\mapsto b_i$ extends uniquely to a monomorphism from $\sfr{n}$ to $\gB$.
\end{lem}
\begin{proof} We need only define the embedding on the atoms of $\sfr{n}$ and this is given by
$$
\{y\} = \left(\bigcap_{j\in y} a_j\right)\cap\left(\bigcap_{j\not\in y}a^c_{j}\right)\mapsto \left(\bigcap_{j\in y} b_j\right)\cap\left(\bigcap_{j\not\in y}b^c_{j}\right).
$$
\end{proof}
\begin{lem}\label{2911113} Let $\gA$ be a subalgebra of a finite Boolean algebra $\gB$. Let $f$ be $\gA$-good and $g$ be $\gB$-good. Let $m$ be the number of atoms of $\gA$ and $n$ the number of atoms of $\gB$. Then there exists an embedding $F:\sfr{m}\rightarrow \sfr{n}$ such that 
\begin{equation}\label{1009121}
g\restriction \gA= F\circ f.
\end{equation} 
\end{lem}
\begin{proof} Let $F':f[\gA]\rightarrow g[\gB]$ be the map $g\circ f^{-1}$. By Lemma \ref{2911111} we see that
$$
\bigcup_{a\in f[\gA]}F'(a) = \bigcup_{a\in f[\gA]}g\circ f^{-1}(a) = g\circ f^{-1}(\bigcup_{a\in f[\gA]}a) = g({1}_\gA) = g({1}_\gB) = {1}_{\sfr{n}},
$$
and so the map $F'$ sends the $*$-free generators of $\sfr{m}$ to $*$-free members of $\sfr{n}$. By Lemma \ref{2911112} we can find an embedding $F:\sfr{m}\rightarrow \sfr{n}$ which agrees with $F'$ on $f[\gA]$. 
\end{proof}
\begin{proof}[Proof of Theorem \ref{2911114}.]\label{proofofmain} Everything on direct limits here is taken from \cite[Pages 49-51]{hodges93}. Fix a countable Boolean algebra $\gA$ let $(\gA_i)_{i\in\scr{N}}$ be a sequence of finite subalgebras of $\gA$ such that $\gA_i\subseteq \gA_{i+1}\subseteq \gA$. For each $i$, let $n_i=|\mathrm{atoms}(\gA_i)|$ and, by choosing the $\gA_i$ appropriately, see to it that $n_i<n_{i+1}$. For each $i$, let $\gC_i = \sfr{n_i}$ and let $f_i$ be an $\gA_i$-good map. For $i<j$, let $f_{i,j}: \gC_i\rightarrow \gC_j$ be the embeddings promised by Lemma \ref{2911113}, with respect to the good maps $f_i$. If $i=j$ then we let $f_{i,j} = \mathrm{Id}$ in $\gA_i$. Now suppose that $i\leq j\leq k$ and let $a_1,...,a_l$ be the $*$-free generators of $\gC_i$. By applying (\ref{1009121}) appropriately, it is straightforward to compute that both $f_{i,k}$ and $f_{j,k}\circ f_{i,j}$ map $a_m$ to $f_k(a_m)$, for each $m\in [l]$. Thus both these embeddings map the $*$-free generators of $\gC_i$ to the same $*$-free members of $\gC_k$ and so, by the uniqueness part of Lemma \ref{2911112}, we see that 
$$
f_{i,k} = f_{j,k}\circ f_{i,j}.
$$  
This shows that $((\scr{N},\leq),(\gC_i)_{i\in \scr{N}},(f_{i,j})_{i,j\in \scr{N}})$ is a directed system.
Let $\gB$ be the corresponding direct limit and let $g_i:\gC_i\rightarrow \gB$ be the corresponding limit maps. We have the following commutative diagram for $i\leq j$:
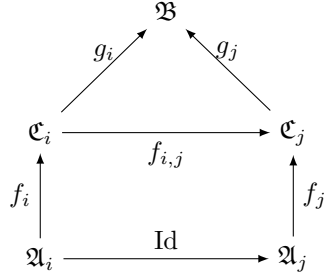
\begin{center}
\begin{tikzpicture}[scale = 0.5]
\matrix [column sep=1.1cm, row sep=1.1cm]
{
\node[]() {}; & \node[](gb) {$\gB$}; &  \node[]() {};\\
\node[](gb1) {$\gC_i$}; & \node[]() {}; & \node[](gb3) {$\gC_j$};\\
\node[](a1) {$\gA_i$}; & \node[]() {} ;  & \node[](a3) {$\gA_j$} ;\\
};
\draw [ ->] (gb1) -- (gb)
node [above,midway] {$g_i$};
\draw [ ->] (gb3) -- (gb)
node [above,midway] {$g_{j}$};
\draw [ ->] (gb1) -- (gb3)
node [below,midway] {$f_{i,j}$};
\draw [ ->] (a1) -- (gb1)
node [left,midway] {$f_i$};
\draw [ ->] (a3) -- (gb3)
node [right,midway] {$f_j$};
\draw [ ->] (a1) -- (a3)
node [above,midway] {$\mathrm{Id}$};
\end{tikzpicture}
\end{center}
Set $\gf(a) = (g_i\circ f_i)(a)$ for any $i$ such that $a\in \gA_i$. Let us now check that $\gf$ satisfies the desired properties. The fact that $\gf$ is injective follows since each $g_{i}$ is an embedding (and in particular injective), and each $f_{i}$ is a $\gA_i$-good map (and in particular injective). Properties (T.\ref{TT2}) and (T.\ref{TT3}) follow by the properties of good maps. Property (T.\ref{TT1}) follows since for every $b\in \gB$, we can find a finite subalgebra $\gA'\subseteq \gA$, such that $b\in \La\gf[\gA]\Ra$.\\\\
Let $\mu:\gA\rightarrow \scr{R}$ be any functional.  By the final part of Lemma \ref{2911111} for each $i$ we can find a unique measure $\lambda_i:\gC_i\rightarrow \scr{R}$ such that $(\forall a\in \gA_i)(\mu(a) = \lambda_i(f_i(a)))$. We now define the measure $\lambda:\gB\rightarrow \scr{R}$ by
$$
\lambda(b) = \lambda_i(g^{-1}_i(b))
$$
for any $i$ such that $b\in \mathrm{ran}(g_i)$. To see that this is well defined we just notice that for $i\leq j$, the uniqueness of $\lambda_i$ implies that $\lambda_{j}\circ f_{i,j} = \lambda_i$.\\\\
Suppose now that $\g{D}$ is a Boolean algebra and $\g{g}:\gA\rightarrow \g{D}$ is an injective map satisfying (T.\ref{TT1}), (T.\ref{TT2}) and (T.\ref{TT3}). Let $(\gA_i)_{i\in\scr{N}}$ and $(n_i)_{i\in\scr{N}}$ be as above. For each $i\in \scr{N}$, let $\g{D}_i = \g{g}[\gA_i]$, $\g{g}_i = \g{g} \restriction \gA_i$ and $p_i:\g{D}_i\rightarrow \sfr{n_i}$ be any isomorphism which injectively maps the $*$-free generators of $\g{D}_i$ to the $*$-free generators of $\sfr{n_i}$. For each $i$, let $f_{i,i+1} = p_{i+1}\circ p_{i}^{-1}$ and $f_i = p_i\circ \g{g}_i$. For $i < j$ let, $f_{i,j} = f_{j-1,j}\circ \cdots f_{i+1,i+2} \circ f_{i,i+1}$ and $f_{i,i} =\mathrm{Id}$. The system $((\scr{N},\leq),(f_{i,j})_{i,j\in \scr{N}},(\sfr{n_i})_{i\in \scr{N}})$ is a directed system. Let $\gB$ be its direct limit and $g_i$ be the corresponding limit maps. Define $\gf:\gA\rightarrow \gB$ by $\gf(a) = (g_i\circ f_i)(a)$, for any $i$ such that $a\in \gA_i$. As above, we see that $\gB$ and $\gf$ satisfy the properties in the statement of Theorem \ref{2911114}, with respect to $\gA$. Finally, for each $i\in \scr{N}$, let $h_i = p_i^{-1}$ and notice that by construction, for $i \leq j$ we have $h_i = h_j\circ f_{i,j}$. Thus we can find an isomorphism $F:\g{B}\rightarrow \g{D}$ such that $h_i = F\circ g_i$, for each $i$. In particular, given a functional $\mu$ on $\gA$, if we let $\lambda$ be the signed measure on $\gB$ defined by $(\forall a)(\mu(a) = \lambda(\gf(a)))$, then we can define a signed measure on $\g{D}$ by $\varphi(a) = \lambda(F^{-1}(a))$, and for each $a\in \gA$ we have,
\begin{eqnarray*}
\mu(a) = \lambda(\gf(a)) = \lambda((g_i\circ f_i)(a)) = \lambda((F^{-1}\circ h_i \circ f_i)(a)) &=& \lambda((F^{-1}\circ p_i^{-1} \circ f_i)(a))\\
&=& \lambda((F^{-1}\circ \g{g}_i)(a))\\
&=& \varphi(\g{g}(a)).
\end{eqnarray*}

\end{proof}
As we have already mentioned, even for very simple submeasures the corresponding measure obtained from Theorem \ref{2911114} may be signed and, even worse, unbounded, as the following example shows.
\begin{exmm}\label{unbounded} For each $n$ let $\mu:\cP([n])\rightarrow \scr{R}$ be the submeasure $\mu([n]) = 1$, $\mu(0) = 0$ and $\mu(a) = \frac{1}{2}$, 
otherwise. If $\lambda:\sfr{n}\rightarrow \scr{R}$ is the corresponding measure from Lemma \ref{2911111} then 
$$
(\exists a\in \sfr{n})(\lambda(a) = -\frac{1}{2}{n \choose 2})
$$
In particular, in the context of Theorem \ref{2911114} and its proof, if we take $\mu:\gA\rightarrow \scr{R}$ to be $(\forall a\in \gA\setminus\{0,1\})(\mu(a) = \frac{1}{2})$ and $\mu(1) = 1$ then for each $i$, the algebra $\La\gf[\gA_i]\Ra$ will contain an element of  $\lambda$-measure $ -\frac{1}{2}{n_i \choose 2}$. Thus $\inf_{b\in\gB}\lambda(b) = -\infty$.
\end{exmm}
\begin{proof} Notice that $I\in \cP([n])^+$ we have $\{I\} = (\bigcap_{i\in I}a_i)\setminus (\bigcup_{i\not \in I}a_i)$. Notice also that $\lambda(\{\{i\}\}) = \lambda(\bigcup_{l\in [n]}a_l) - \lambda(\bigcup_{l\in [n]\setminus \{i\}}a_l)$. We also have that for $i\neq j$ we have $\lambda(\{\{i,j\}\})  = \lambda(\bigcup_{l\in [n]}a_l) - (\lambda(\bigcup_{l\in [n]\setminus\{i,j\}} a_l) + \lambda(\{\{i\}\})+\lambda(\{\{j\}\}))$. This shows that if $i\neq j$ then 
$$
\lambda(\{\{i,j\}\}) = -\lambda(\bigcup_{l\in [n]}a_l)-\lambda(\bigcup_{l\in [n]\setminus \{i,j\}}a_l) + \lambda(\bigcup_{l\in [n]\setminus \{i\}} a_l)+\lambda(\bigcup_{l\in [n]\setminus \{j\}}a_l)).
$$
Thus we can take $a = \{y:y \in [[n]]^2\}$.\end{proof}
It is not clear to us how to predict when a submeasure will generate a non-negative measure, or even just a bounded signed measure.\\\\
Let us now define explicitly an instance of Theorem \ref{2911114}. Let $(X_i)_{i\in\scr{N}}$ be a sequence of finite non-empty sets. Let $X^{(n)} = \prod_{i\in [n]}X_i$ and $X = \prod_{i\in\scr{N}}X_i$. For convenience we assume that for each $i$, we have 
\begin{equation}\label{1309123}
|X_i|>1,
\end{equation} and we also set $X^{(0)} = \{\emptyset\}$. Let $T_1 = \cP(X_1)^+$ and
\begin{equation}\label{Tidef}
T_{i+1} = \{A\subseteq X^{(i+1)}:(\forall t\in X^{(i)})(\exists s\in A)(s\restriction [i] = t)\}.
\end{equation}
Let $T^{(n)} = \prod_{i\in [n]} T_i$ and $T = \prod_{i\in\scr{N}} T_i$. Let $\gA = \mathrm{Clopen}(X)$ and $\gB= \mathrm{Clopen}(T)$.
\begin{df}\label{explicitf} Let $\gf:\gA\rightarrow \gB$ be defined as follows. For every $n\in \scr{N}$ and $t\in X^{(n)}$ we set
\begin{eqnarray*}
\gf([t]) &=& \bigcup\{[f]: f\in T^{(n)}\wedge (\forall i\in [n])(t\restriction [i]\in f(i))\}\\
&=&\{f\in T: (\forall i\in [n])(t\restriction [i]\in f(i))\}
\end{eqnarray*}
We then extend $\gf$ to all members of $\gA$ by taking unions.
\end{df}
\begin{prop}\label{1309121} The function $\g{f}$ is injective and satisfies (T.\ref{TT2}) and (T.\ref{TT3}) of Theorem \ref{2911114}.
\end{prop} 
Before we prove Proposition \ref{1309121}, it will be helpful to record the following.
\begin{df} For $f\in T^{(n)}$ say that $t\in X^{(n)}$ \eb{generates} $f$\index{generates@$t$ generates $f$} if and only if
$$
(\forall i\in[n])(t\restriction [i]\in f(i)).
$$
\end{df}
\begin{lem}\label{1309122} For every $n\in \scr{N}$ and $f\in T^{(n)}$, there exists $t\in X^{(n)}$ that generates $f$. Conversely, for every $n$ and $A\subseteq X^{(n)}$ there exists an $f\in T^{(n)}$ that is generated by precisely the members of $A$.
\end{lem} 
\begin{proof}  The first claim may be seen by induction on $n$ using (\ref{Tidef}). Indeed, for the case $n=1$, any member of $f(1)$ generates $f$. Suppose it is true for $n$ and let $f\in T^{(n+1)}$. By induction, find $t\in X^{(n)}$ that generates $f\restriction [n]$. By (\ref{Tidef}), there exists $s\in f(n+1)$ such that $s\restriction [n] = t$ and so, since $t$ generates $f\restriction [n]$, it must be the case that $s$ generates $f$. The second claim also proceeds by induction on $n$. For the case $n=1$, if $A\subseteq X^{(1)}$ then the function $\{(1,A)\}\in T^{(1)}$ and is generated precisely by $A$. Now suppose it is true for $n$, and let $A\subseteq X^{(n+1)}$. Let $g\in T^{(n)}$ be generated by precisely the members of $B:=\{t\restriction [n]:t\in A\}$. Now fix $x\in X_{n+1}$ and let $f = g^\frown (A\cup \{t^\frown x: t\in X^{(n)}\setminus B\})$. It is clear that $f\in T^{(n+1)}$. If $t\not\in A$, then $t\restriction [n]\not\in B$ and so $t\restriction [n]$ cannot generated $g$, which means that $t$ cannot generate $f$. If $t\in A$, then by definition $t\in f(n+1)$, and since $t\restriction [n]$ generates $g$, we must have that $t$ generates $f$.
\end{proof}
\begin{proof}[Proof of Proposition \ref{1309121}] This all follows from Lemma \ref{1309122}. For injectivity, Let $n\in \scr{N}$ and suppose that $C,B\subseteq X^{(n)}$, such that $C\neq B$. Without loss of generality, we can find $t\in C\setminus B$. Now let $f\in T^{(n)}$ be generated only by $t$. Then $f\in \bigcup_{s\in C}\g{f}([s])\setminus \bigcup_{s\in B}\g{f}([s])$. For property (T.\ref{TT2}), it is enough to check that for each $n\in \scr{N}$, the collection $\{\gf([t]):t\in X^{(n)}\}$ forms a $*$-free collection in $\gB$. For this just observe that if $A$ is a non-empty subset of $X^{(n)}$ and $f\in T^{(n)}$ is generated by precisely the members of $A$, then 
$$
f\in \left(\bigcap_{t\in A}\gf{[t]}\right)\cap\left(\bigcap_{t\in A}\gf{[t]}\right)^c\neq \emptyset.
$$
Property (T.\ref{TT3}) follows from how we constructed $\gf$ (by taking unions). \end{proof}
As a final remark, notice that if $f:\gC\rightarrow \gC'$ is such that we always have $f(c\cup d) = f(c)\cup f(d)$ then for any (non-negative) measure $\lambda$ on $\gC'$ one can define a submeasure $\mu$ on $\gC$ by $\mu(c) = \lambda(f(c))$. This raises the following question: If $\lambda$ is the Lebesgue measure on $\gB$, then what submeasure do we get on $\gA$? The counting required to understand this submeasure $\mu$ (say), for general $X_i$, is quite complicated. In the case where we restrict to $X_i = [2]$, the counting becomes manageable, and it can be shown that $\mu$ is not pathological and is not exhaustive, witnessed by an antichain of Borel sets of length continuum, once we have extended $\mu$ to $\mathrm{Borel}(X)$ as in (\ref{100}).\footnote{See \cite[\S 7]{phd1}.}
\begin{proof}[Proof of Lemma \ref{0612111}]\label{proofofmatrix} By induction on $n$. For the case $n=1$, the matrix in question is the identity, so let us show that this is true for $n+1$ assuming it is true for $n$. Let $m = 2^n-1$ and $m' = 2^{n+1}-1$. Enumerate $\cP([n+1])^+ = \{y_i:i\in [m']\}$ so that $\{y_i:i\in [m]\}$ is an enumeration of $\cP([n])^+$, $y_{m+1} = \{n\}$ and $y_{i+m + 1} = y_i\cup\{n\}$ for $i\in [m]$. Let $A_n$ be $m\times m$ matrix where, for $i,j\in [m]$, we let
$$
A_n(i,j) = \left\{\begin{array}{cl} 1,\quad&\mbox{if $y_{i}\cap y_{j}\neq 0$;}\\
0,\quad&\mbox{otherwise.}\end{array}\right.
$$
Let $A_{n+1}$ be the $m'\times m'$ matrix where, for $i,j\in [m']$, we set
$$
A_{n+1}(i,j) = \left\{\begin{array}{cl} 1,\quad&\mbox{if $y_{i}\cap y_{j}\neq 0$;}\\
0,\quad&\mbox{otherwise.}\end{array}\right.
$$
We want to show that $A_{n+1}$ is invertible. By induction the rows of $A_n$ are linearly independent. Let $v_i$ denote the $i$th row of $A_{n+1}$ and $u_i$ the $i$th row of $A_n$. Then
$$
v_i =  \left\{\begin{array}{cl} {u_i}^\frown 0^\frown {u_i},\quad&\mbox{if $i\in [m]$;}\\
{0^{m}}^\frown 1^{m'-m},\quad&\mbox{if $i = m+1$;}\\
{u_{i-m-1}}^\frown 1^{m'-m},\quad&\mbox{otherwise.} \end{array}\right.
$$
That is 
$$
A_{n+1} = 
\left( \begin{matrix}
  A_n & (0^m)^T & A_n \\
  0^m & 1 & 1 \\
  A_n & 1 & 1
 \end{matrix}\right).
$$
Here, $(0^m)^T$ denotes the column vector of length $m$ containing only $0$'s. Now let $\lambda_i\in \scr{R}$ be such that $\sum_{i\in [m']}\lambda_i v_i = 0^{m'}$. Since
$$
\sum_{i\in [m]}\lambda_i u_i+\sum_{i=m+2}^{m'}\lambda_{i}u_{i-m-1} = \sum_{i\in [m]}(\lambda_i + \lambda_{i+m+1})u_i =  0^m,
$$
by the linear independence of the $u_i$, we must have
\begin{equation}\label{0909121}
(\forall i\in [m])(\lambda_{i+m+1} = -\lambda_i).
\end{equation} 
Considering the $(m+1)$th column of $A_{n+1}$, by (\ref{0909121}), we see that $\lambda_{m+1}-\sum_{i\in [m]}\lambda_{i} = 0$. We now have
\begin{eqnarray*}
0= \sum_{i\in [m']}\lambda_i v_i &=& \sum_{i=1}^m \lambda_i v_i + \lambda_{m+1}v_{m+1} + \sum_{i=m+2}^{m'} \lambda_i v_i\\
&=& \sum_{i=1}^m \lambda_i v_i + \sum_{i = 1}^m\lambda_{i} {0^m}^\frown 1^{m'-m}+ \sum_{i=1}^{m} \lambda_{i+m+1} v_{i+m+1}\\
&=& \sum_{i=1}^m \lambda_i u_i^\frown 0^\frown u_i + \sum_{i = 1}^m\lambda_{i} {0^m}^\frown 1^{m'-m} - \sum_{i=1}^{m} \lambda_{i} u_i^\frown 1^{m'-m}\\
&=& \sum_{i=1}^m \lambda_i {0^{m+1}}^\frown u_i + \sum_{i = 1}^m\lambda_{i} {0^m}^\frown 1^{m'-m} - \sum_{i=1}^{m} \lambda_{i} {0^m}^\frown 1^{m'-m}\\
&=& \sum_{i\in [m]}\lambda_i{0^{m+1}}^\frown u_i.
\end{eqnarray*}
By the linear independence of the $u_i$ and (\ref{0909121}), we may conclude that $\lambda_i = 0$ for each $i\neq m+1$. But then 
$$
\lambda_{m+1} {0^{m}}^\frown 1^{m'-m} = 0^{m'}
$$
and so we must have $\lambda_{m+1} = 0$, also. Thus the rows $\{v_i:i\in [m']\}$ are linearly independent and $A_{n+1}$ is invertible.\end{proof}

\section{Acknowledgements} The results presented here are taken from \cite{phd1} and accordingly the author wishes to thank his Ph.D. supervisor Professor Mirna D\v zamonja for her help. The author would also like to thank Professor Grzegorz Plebanek for his useful remarks, in particular the proof of Proposition \ref{gdelta} as it appears here is due to him.
\noindent
\bibliographystyle{alpha}
\bibliography{refs}
\end{document}